\documentclass[11pt,thmsa]{article}%
\usepackage{amsmath}
\usepackage{amsfonts}
\usepackage{amssymb}
\usepackage{graphicx}%
\newtheorem{theorem}{Theorem}[section]

\newtheorem{problem}[theorem]{Problem}

\newtheorem{corollary}[theorem]{Corollary}

\newtheorem{lemma}[theorem]{Lemma}

\newtheorem{proposition}[theorem]{Proposition}

\newenvironment{proof}[1][Proof]{\noindent\textbf{#1.} }{\ \rule{0.5em}{0.5em}}

\begin{document}
\title{Towards the classification of odd dimensional  homogeneous  reversible Finsler spaces with positive flag curvature \thanks{Supported by NSFC (no. 11271216, 11271198, 11221091), State Scholarship Fund of CSC (no. 201408120020), SRFDP of China, Science and Technology Development Fund for Universities and Colleges in Tianjin (no. 20141005), and Doctor fund of Tianjin Normal University (no. 52XB1305) }}
\author{Ming Xu$^1$, Shaoqiang Deng$^2$\thanks{Corresponding author} \\
\\
$^1$College of Mathematics\\
Tianjin Normal University\\
 Tianjin 300387, P. R. China\\
 Email: mgmgmgxu@163.com.\\
 \\
$^2$School of Mathematical Sciences and LPMC\\
Nankai University\\
Tianjin 300071, P. R. China\\
E-mail: dengsq@nankai.edu.cn}

\date{}


\maketitle

\begin{abstract}
In this paper, we use the flag curvature formula for homogeneous Finsler spaces in our previous work to classify odd dimensional smooth coset spaces admitting positively curved reversible homogeneous Finsler metrics.
We will show that the most features of L. B\'{e}rard-Bergery's classification results for odd dimensional positively curved Riemannian homogeneous spaces can be generalized to reversible Finsler spaces.



\end{abstract}


\section{Introduction}

Finding new examples of compact manifolds admitting Riemannian metrics of positive sectional curvature is one of the central problems in Riemannian geometry. In the  homogeneous setting, the problem is to   classify   positively curved Riemannian homogeneous spaces, and this has been achieved in several classical works in this field; See \cite{Ber61}, \cite{Wallach1972}, \cite{AW75} and \cite{BB76}.
Notice that in \cite{Ber61}, M. Berger missed one in his classification of positively curved normal homogeneous spaces, as pointed out by B. Wilking in \cite{Wi1999}. In the classification of odd dimensional positively curved Riemannian homogeneous spaces by L. B\'{e}rard Bergery \cite{BB76}, a gap was recently found  by the first-named author of this paper and J.A. Wolf, and it has been   corrected by B. Wilking;  See \cite{XW2015}.  Based on some methods developed in \cite{Wi2006}, B. Wilking and W. Ziller provided an alternative and modern proof of the classification  in \cite{BB76} in  their  recent preprint \cite{WZ2015}.

In homogeneous Finsler geometry,  the following  problem  is  of great  significance:

\begin{problem}\label{main-problem}
Classify the smooth coset spaces $G/H$ admitting a $G$-invariant Finsler metric of positive flag curvature.
\end{problem}

For simplicity, we will call a homogeneous space {\it positively curved}
when it admits an invariant Finsler metric of positive flag curvature, or if  it has been endowed with such a metric. By the Bonnet-Myers Theorem for Finsler spaces, a positively curved homogeneous space must be compact.

Problem \ref{main-problem} was first studied by S. Deng and Z. Hu in \cite{HD11}, where they classified homogeneous Randers metrics with positive flag curvature and vanishing S-curvature. Note that their classification is also valid for homogeneous $(\alpha,\beta)$-spaces with positive flag curvature and vanishing S-curvature \cite{XD2015}.

Recently  big progress has been made on the classification with more generality.
In \cite{XD2014},  the authors of this paper  classified positively curved
normal homogeneous Finsler spaces, generalizing the classical results of \cite{Ber61}. In the joint work of the authors with L. Huang and Z. Hu \cite{XDHH2014}, we classified even dimensional positively curved homogeneous Finsler spaces, generalizing the results of \cite{Wallach1972}.

It should be noted that a very useful flag curvature formula for homogeneous Finsler spaces has been established in \cite{XDHH2014} (see Theorem \ref{flag-curvature-formula-thm} below). In this paper, we will apply this formula to  the classification of odd dimensional positive curved homogeneous Finsler spaces.

The general theme for the classification has been set up in \cite{XD2014}. Recall that  for a positively curved homogeneous Finsler space $(G/H,F)$ with a bi-invariant orthogonal decomposition $\mathfrak{g}=\mathfrak{h}+\mathfrak{m}$ for the compact Lie group $\mathfrak{g}$, and a fundamental Cartan subalgebra $\mathfrak{t}$ of $\mathfrak{g}$ (i.e., $\mathfrak{t}\cap\mathfrak{h}$ is a
Cartan subalgebra of $\mathfrak{h}$), we divide our discussion into the following three cases:
\begin{description}
\item{\bf Case I.} Each root plane of $\mathfrak{h}$ is  a root plane of $\mathfrak{g}$.
\item{\bf Case II.} There exist two roots $\alpha$ and $\beta$  of $\mathfrak{g}$ from different simple factors, such that $\mathrm{pr}_\mathfrak{h}(\alpha)=\mathrm{pr}_\mathfrak{h}(\beta)=\alpha'$
    is a root of $\mathfrak{h}$.
\item{\bf Case III.} There exists a linearly independent pair of roots $\alpha$ and $\beta$ of $\mathfrak{g}$ from the same simple factor, such that $\mathrm{pr}_\mathfrak{h}(\alpha)=\mathrm{pr}_\mathfrak{h}(\beta)=\alpha'$
    is a root of $\mathfrak{h}$.
\end{description}

The classification is only up to local isometry. So we introduce the definition of equivalence (see Subsection 2.5) for coset spaces to specify some typical procedures which results local isometries, such as changing $G$ to its covering group, changing $H$ while fixing its identity component, cancelling common product factors from $G$ and $H$, replacing $H$ with $\sigma(H)$, where $\sigma$ is  an isomorphism of $G$, and so on. This method greatly reduces the complexity of the statement and the proofs of the classification.

In this paper we shall consider the classification of odd dimensional reversible homogeneous Finsler spaces with positive flag curvature. Our motivation to consider reversible metrics is
twofold. On one hand,  in our practical use of the flag curvature formula in Theorem \ref{flag-curvature-formula-thm}, we have found that the discussion will be much simpler under the assumption that the metric is reversible. On the other hand, restriction to reversible Finsler metrics will not lose much  generality, and  contains all the Riemannian ones. In particular, with some more detailed discussion according to the remark in Subsection 6.3, the classification result in this paper can cover that in the Riemannian case given by L. B\'{e}rard-Bergery.

The main results of this paper are   Theorems \ref{mainthm-part-1},  \ref{mainthm-part-2} and  \ref{mainthm-part-3}, and  these theorems can be summarized as the following

\begin{theorem} \label{mainthm}
Let $(G/H,F)$ be an odd dimensional positively curved reversible homogeneous Finsler space. Then we have the following:
\begin{description}
\item{\rm (1)} If it belongs to Case I,  then up to equivalence, either $G$ is a compact simple Lie group or $G/H$ is one of the homogeneous spheres $S^{2n-1}=\mathrm{U}(n)/\mathrm{U}(n-1)$ and
    $S^{4n-1}=\mathrm{Sp}(n)\mathrm{U}(1)/\mathrm{Sp}(n-1)\mathrm{U}(1)$, $n>1$, or
    the $\mathrm{U}(3)$-homogeneous Aloff-Wallach's spaces.
\item{\rm (2)} If it belongs to Case II, then up to equivalence, $G/H$ is one of the homogeneous spheres $S^3=\mathrm{SO}(4)/\mathrm{SO}(3)$ and
    $S^{4n-1}=\mathrm{Sp}(n)\mathrm{Sp}(1)/\mathrm{Sp}(n-1)\mathrm{Sp}(1), n>1$,
    or Wilking's space $\mathrm{SU}(3)\times\mathrm{SO}(3)/\mathrm{U}(2)$.
\item{\rm (3)} If it belongs to Case III, then up to equivalence, $G/H$ is one of the homogeneous spheres
    $S^{2n-1}=\mathrm{SO}(2n)/\mathrm{SO}(2n-1), n>2$, $S^7=\mathrm{Spin}(7)/\mathrm{G}_2$, and
    $S^{15}=\mathrm{Spin}(9)/\mathrm{Spin}(7)$,
    or one of the Berger's spaces
    $\mathrm{SU}(5)/\mathrm{Sp}(2)\mathrm{U}(1)$ and  $\mathrm{Sp}(2)/\mathrm{SU}(2)$.
\end{description}
\end{theorem}

Notice that any invariant Finsler metric on the coset space $S^{2n-1}=\mathrm{SO}(2n)/\mathrm{SO}(2n-1)$ or $S^7=\mathrm{Spin}(7)/\mathrm{G}_2$ must be the standard Riemannian metric of positive constant curvature. On the other hand, as pointed out in \cite{HD11} and \cite{XD2015}, the Aloff-Wallach's spaces admit non-Riemannian homogeneous Randers metrics or $(\alpha,\beta)$-metrics with positive flag curvature and vanishing S-curvature. Moreover, any of the other coset spaces listed in Theorem \ref{mainthm} admits a non-Riemannian positively curved normal homogeneous Finsler metric; see \cite{XD2014}. Though it is not clearly stated in these literatures, the reversibility can be easily fulfilled for each of the non-Riemannian cases.

 Theorems \ref{mainthm-part-1},  \ref{mainthm-part-2} and  \ref{mainthm-part-3} will be proved separately in the sequel.
 These results cover most classification results in \cite{BB76}. However,  the classification of this paper is not complete for Case I when $G$ is compact simple,
where  the homogeneous spheres $\mathrm{SU}(n)/\mathrm{SU}(n-1)$,
$\mathrm{Sp}(n)/\mathrm{Sp}(n-1)$, and the $\mathrm{SU}(3)$-homogeneous
Aloff-Wallach's spaces, as well as some other possible candidates,  are known to have invariant positively curved Finsler metrics (\cite{HD11}).
One reason that our classification cannot be perfect in this case is as the following.
The method in this paper originates from the most traditional algebraic one, namely,  to prove a homogeneous space cannot  be positively curved, we try to find  a linearly independent and commutative pair $u$ and $v$ from $\mathfrak{m}$ such that the sectional (or flag) curvature for the plane spanned by them (the flag pole needs to be specified for the flag curvature) vanishes. But this method is not valid for some rare cases in Case I; see \cite{XW2015}.

In Section 2, we give a brief summary of basic notions in Finsler geometry and homogeneous Finsler geometry, and define the notion of equivalence which will be used throughout this paper. In Section 3, we present the general theme for the classification of odd dimensional positively curved homogeneous Finsler spaces, including the flag curvature formula, the rank equality, and some useful lemmas. In Sections 4 and 5, we discuss the classification of odd dimensional positively curved reversible homogeneous Finsler spaces in Case III. In Section 6, we discuss the classification of odd dimensional positively curved reversible homogeneous Finsler spaces in Case II and I. Section 7 is an appendix where we summarize the presentation of root systems of compact simple Lie algebras  used in this paper.

 We are grateful to J.A. Wolf, W. Ziller, B. Wilking and L. Huang  for helpful discussions. The first author thanks the Department of Mathematics at the University of California, Berkeley, for hospitality during the preparation of this paper.

\section{Preliminaries}

In this section, we summarize some definitions and fundamental results in Finsler geometry;
See \cite{CS04} and \cite{DE12} for more details. In this paper, we will only consider connected smooth manifolds and connected Lie groups.

\subsection{Minkowski norm and Finsler metric}

A {\it Minkowski norm} on a real vector space $\mathbf{V}$, $\dim\mathbf{V}=n$,
is a continuous real-valued function $F:\mathbf{V}\rightarrow[0,+\infty)$
satisfying the following conditions:
\begin{description}
\item{\rm (1)}\quad $F$ is positive and smooth on $\mathbf{V}\backslash\{0\}$;
\item{\rm (2)}\quad $F(\lambda y)=\lambda F(y)$ for any $\lambda >0$;
\item{\rm (3)}\quad With respect to any linear coordinates $y=y^i e_i$,
the Hessian matrix
\begin{equation}
(g_{ij}(y))=\left(\frac12[F^2]_{y^i y^j}\right)
\end{equation}
is positive definite at any nonzero $y$.
\end{description}

The Hessian matrix $(g_{ij}(y))$ and its inverse $(g^{ij}(y))$ can be used
to move up and down indices of relevant tensors in Finsler geometry.

Given a nonzero vector $y$, the Hessian matrix $(g_{ij}(y))$ defines an inner product
$\langle\cdot,\cdot\rangle_y$ on $\mathbf{V}$ by
\begin{equation*}
\langle u,v\rangle_y=g_{ij}(y)u^i v^j,
\end{equation*}
where $u=u^i e_i$ and $v=v^i e_i$. In the literature, the above inner
product is also denoted as $\langle\cdot,\cdot\rangle_y^F$ to specify the norm. Sometimes it is shortened as $g_y$ or $g_y^F$. It is obvious that the inner product can also be expressed as
\begin{equation}
\langle u,v\rangle_y=\frac12\frac{\partial^2}{\partial s\partial t}
[F^2(y+su+tv)]|_{s=t=0}.
\end{equation}
It is easy to check that the above definition is independent of the choice of
linear coordinates.

Let $M$ be a smooth manifold with dimension $n$. A {\it Finsler metric} $F$ on
$M$ is a continuous function $F:TM\rightarrow[0,+\infty)$ which is
positive and smooth on the slit tangent bundle $TM\backslash 0$, and whose
restriction to each tangent space is a Minkowski norm. Generally, $(M,F)$
is called a {\it Finsler manifold} or a {\it Finsler space}.

Here are some important examples.

Riemannian metrics are a special class of Finsler metrics such that
the Hessian matrix only depend on $x\in M$.
For a Riemannian manifold, the metric is often referred to as the global
smooth section $g_{ij}dx^i dx^j$ of $\mathrm{Sym}^2(T^* M)$.
Unless otherwise stated, we mainly deal with non-Riemannian metrics in this paper.

Randers metrics are the simplest and the most important class of
non-Riemannian metrics in Finsler geometry. A Randers metric can be written as
$F=\alpha+\beta$, where $\alpha$ is a Riemannian metric and
$\beta$ is a 1-form. The notion of Randers metrics can be
naturally generalized to $(\alpha,\beta)$-metrics. An
$(\alpha,\beta)$-metric is a Finsler metric of the form
$F=\alpha\phi(\beta/\alpha)$, where $\phi$ is a positive smooth real
function, $\alpha$ is a Riemannian metric and $\beta$ is a 1-form. In
recent years, there have been a lot of research works concerning
$(\alpha,\beta)$-metrics as well as Randers metrics.

Recently, we have defined and studied $(\alpha_1,\alpha_2)$-metrics and
introduced the more generalized class of $(\alpha_1,\alpha_2,\ldots,\alpha_k)$-metrics; see \cite{DX2014} and \cite{XDHH2014}.
Such metrics naturally appear in the study of homogeneous Finsler geometry.

A Minkowski norm or a Finsler metric is called {\it reversible} if $F(y)=F(-y)$
for any $y\in\mathbf{V}$ or $F(x,y)=F(x,-y)$ for any $x\in M$ and $y\in T_xM$.
Obviously, a Riemannian metric is reversible, and
a non-Riemannian Randers metric must be non-reversible. Note that a non-Riemannian  $(\alpha,\beta)$-metric is reversible if  the function $\phi$ is an even function, and there exist many non-reversible $(\alpha,\beta)$-metrics.

\subsection{Geodesic spray and geodesic}

Let   $(M,F)$ be a Finsler space. A local coordinate system $\{ x=(x^i)\in M;
y=y^j\partial_{x^j}\in T_x M\}$ on $TM$ is called a {\it standard local
coordinates system}.
The geodesic spray is a vector field $\mathbf{G}$ globally defined on $TM\backslash 0$. On a
standard local coordinate system, it can be expressed as
\begin{equation}
\mathbf{G}=y^i\partial_{x^i}-2\mathbf{G}^i\partial_{y^i},
\end{equation}
in which
\begin{equation}
\mathbf{G}^i=\frac14 g^{il}([F^2]_{x^k y^l}y^k-[F^2]_{x^l}).
\end{equation}


A non-constant curve $c(t)$ on $M$ is called a geodesic if
$(c(t),\dot{c}(t))$ is an integration curve of ${G}$,
in which the tangent field $\dot{c}(t)=\frac{d}{dt}c(t)$ along the
curve gives the speed. On a  standard local coordinate, a geodesic
$c(t)=(c^i(t))$ can be characterized by  the equations
\begin{equation}
\ddot{c}^i(t)+2\mathbf{G}^i(c(t),\dot{c}(t))=0.
\end{equation}

It is well known that $F(c(t),\dot{c(t)})$ is a constant function, or in other words,
a geodesic defined by the above equations must be  of nonzero constant speed.

\subsection{Riemann curvature and flag curvature}

 In Finsler geometry, there is a similar notion of  curvature as in the Riemannian case, which is called the Riemann curvature. It can be defined either by the Jacobi field or the structure equation for the curvature of the Chern connection.

On a standard local coordinate system, the Riemann curvature is
a linear map ${R}_y=R_k^i(y)\partial_{x^i}\otimes dx^k: T_x M\rightarrow T_xM$,
defined by
\begin{equation}\label{local-coordinate-formula-rieman-curvature}
R_k^i(y)=2\partial_{x^k}{G}^i-y^j\partial^2_{x^j y^k}{G}^i
+2{G}^j\partial_{y^j y^k}^2{G}^i-\partial_{y^j}{G}^i
\partial_{y^k}{G}^j.
\end{equation}
When the metric needs to be specified, the Riemann curvature is denoted
as ${R^F}_y=({R^F})_k^i(y)\partial_{x^i}\otimes dx^k$.
From Proposition 6.2.2 of \cite{Sh2001}, it is easily seen that the Riemann curvature $R_y$ is self-adjoint with respect to
the inner product $\langle\cdot,\cdot\rangle_y$.

Using the Riemann curvature, we can generalize the notion of sectional
curvature to  Finsler geometry, called the flag curvature. Let $y\in T_xM$
be a nonzero tangent vector and $\mathbf{P}$ a tangent plane in $T_xM$
containing $y$, and suppose it is linearly spanned by $y$ and $v$. Then the flag
curvature of the pair  $(y,\mathbf{P})$ is defined  by
\begin{equation}\label{def-flag-curv}
K(x,y,y\wedge v)=K(x,y,\mathbf{P})=
\frac{\langle R_y v,v\rangle_y}{\langle y,y\rangle_y\langle v,v\rangle_y
-\langle y,v\rangle_y^2}.
\end{equation}
Obviously, the flag curvature in (\ref{def-flag-curv}) does not depend on the choice of
$v$ but only on $y$ and $\mathbf{P}$. Sometimes we  also write the flag
curvature of a Finsler metric $F$ as $K^F(x,y,y\wedge v)$ or
$K^F(x,y,\mathbf{P})$ to indicate the metric explicitly.

\subsection{Totally geodesic submanifold}

A submanifold $N$ of a Finsler space $(M,F)$ can be naturally endowed with a submanifold Finsler metric, denoted as $F|_N$. At each point $x\in N$, the Minkowski norm $F|_N(x,\cdot)$
is just the restriction of the Minkowski norm $F(x,\cdot)$ to $T_x N$. We say that $(N,F|_N)$
is a {\it Finsler submanifold} or a {\it Finsler subspace}.

A Finsler subspace $(N,F|_N)$ of $(M,F)$ is called {\it totally geodesic} if any geodesic of
$(N,F|_N)$ is also a geodesic of $(M,F)$. On a standard local coordinate system $(x^i,y^j)$
such that $N$ is locally defined  by $x^{k+1}=\cdots=x^n=0$, the totally geodesic condition can be expressed as
\begin{equation*}
\mathbf{G}^i(x,y)=0, \quad k<i\leq n, x\in N, y\in T_x N.
\end{equation*}
A direct calculation shows that in this case the Riemann curvature $R_y^{F|_N}:T_x N\rightarrow T_x N$ of $(N,F|_N)$ is just the restriction of the Riemann curvature $R_y^F$ of $(M,F)$, where $y$ is a nonzero tangent
vector of $N$ at $x\in N$. Therefore we have

\begin{proposition}\label{prop-2-2}
Let $(N,F|_N)$ be a totally geodesic submanifold of $(M,F)$. Then for any $x\in N$,
 $y\in T_x N\backslash 0$, and a tangent plane $\mathbf{P}\subset T_x N$
containing $y$, we have
\begin{equation}
K^{F|_N}(x,y,\mathbf{P})=K^F(x,y,\mathbf{P}).
\end{equation}
\end{proposition}

As in Riemannian geometry, the local properties of exponential maps implies any connected component $N$ of the common fixed points
for a set of isometries $\{\rho_a,a\in\mathcal{A}\}$
of $(M,F)$ is a totally geodesic sub-manifolds of $(M,F)$. To be more precise, for each point
$x\in N$,
$$T_x N=\{y\in T_x M|{\rho_a}_*y=y,\forall a\in \mathcal{A}\}$$
and $N$ contains a small neighborhood of $x$ in
$\exp_x T_x N$.

\subsection{Homogeneous Finsler geometry}

Let $(M,F)$ be  a connected Finsler manifold.  If the full group $I(M,F)$ of isometries of $(M, F)$ (or equivalently, the identity component $I_0(M,F)$ of $I(M,F)$)
acts transitively on $M$, then we say that $(M, F)$ is a
{\it homogeneous Finsler space},
or $F$ is  a {\it homogeneous Finsler metric}.  It is shown in \cite{DH2004} that $(I(M,F)$
(hence $G=I_0(M,F)$) is a Lie transformation group on $M$. Let $H$ be the compact isotropic subgroup of $G$ at a point $o\in M$. Then $M$ is diffeomorphic to the smooth coset space
$G/H$, associated with a canonical smooth projection map
$\pi:G\rightarrow M=G/H$ such that $\pi(e)=o$. The tangent space
$T_o M$ can be naturally identified with $\mathfrak{m}=\mathfrak{g}/\mathfrak{h}$, in which $\mathfrak{g}$
and $\mathfrak{h}$ are the Lie algebras of $G$ and $H$, respectively.
The isotropy action of $H$ on $T_oM$ coincides with the induced $\mathrm{Ad}(H)$-action on $\mathfrak{m}$. In the cases we will consider in this paper, $\mathfrak{m}$ can be realized as a
complement subspace of $\mathfrak{h}$ in $\mathfrak{g}$ which is preserved
by $\mathrm{Ad}(H)$-actions. Then we have an {\it $\mathrm{Ad}(H)$-invariant decomposition}
$\mathfrak{g}=\mathfrak{h}+\mathfrak{m}$ satisfying the reductive condition
$[\mathfrak{h},\mathfrak{m}]\subset\mathfrak{m}$.

If $(M,F)$ is positively curved, then by the Bonnet-Myers Theorem,
$M$ must be compact, hence $G=I_0(M,F)$ is also compact. Fix a
bi-invariant inner product on $\mathfrak{g}$. Then we can realize $\mathfrak{m}$ as the bi-invariant orthogonal complement of $\mathfrak{h}$.
In this case, the decomposition $\mathfrak{g}=\mathfrak{h}+\mathfrak{m}$
is called a {\it bi-invariant orthogonal decomposition} for the homogeneous space $G/H$.

Notice that for any closed connected subgroup $G$ of $I_0(M,F)$ which acts transitively on $M$, we have a corresponding representation $M=G/H$.
The most typical example is the nine classes of
homogeneous spheres; See \cite{Bo1940}.
For the convenience, we will consider a slightly more general situation, namely, for a positively curved homogeneous Finsler space $M=G/H$, we only require that the Lie algebra $\mathfrak{g}$ of $G$    is compact (i.e., $G$ is quasi-compact). The notion of bi-invariant orthogonal decomposition is still valid
in this case.

To simplify the discussion and avoid  unnecessary iterance in the classification, we will not distinguish homogeneous Finsler spaces which are locally isometric to each other. In particular, we will call $(G_1/H_1,F_1)$ and $(G_2/H_2,F_2)$ (with corresponding bi-invariant orthogonal decompositions for the compact Lie groups $\mathfrak{g}_1$ and $\mathfrak{g}_2$ respectively) {\it equivalent} if one of the following conditions is satisfied
\begin{description}
\item{\rm (1)}\quad $G_1$ is a covering group of $G_2$, $H_1$ has the same identity component as $H_2$, and $F_1$ is naturally induced from $F_2$, up to a positive scalar;
\item{\rm (2)}\quad $G_1=G_2\times G'$, $H_1=H_2\times G'$,  and $F_1$ and $F_2$ are induced from the same Minkowski norm, when $\mathfrak{m}_1$ and $\mathfrak{m}_2$ are naturally identified as the same vector space;
\item{\rm (3)}\quad There exists a group isomorphism from $G_1$ to $G_2$, which maps
$H_1$ onto $H_2$ and induces an isometry from $F_1$ to $F_2$.
\end{description}
 The above notion actually  defines an {\it equivalent relation} on the set of compact homogeneous Finsler spaces $G/H$ with $\mathfrak{g}=\mathrm{Lie}(G)$ compact. In the following, compact homogeneous Finsler spaces in the same equivalent class will not be distinguished. Thus our classification will be local, or in other words, on the Lie algebra level.

\section{The general theme for the classification}
In this section, we establish the theme for our classification.
\subsection{A flag curvature formula for homogeneous Finsler spaces}

In \cite{XDHH2014}, we proved the following theorem:
\begin{theorem} \label{flag-curvature-formula-thm}
Let $(G/H,F)$ be a connected homogeneous Finsler space, and $\mathfrak{g}=\mathfrak{h}+\mathfrak{m}$ be an $\mathrm{Ad}(H)$-invariant
decomposition for $G/H$. Then for any linearly independent commutative pair
$u$ and $v$ in $\mathfrak{m}$ satisfying
$
\langle[u,\mathfrak{m}],u\rangle^F_u=0
$,
we have
\begin{equation*}
K^F(o,u,u\wedge v)=\frac{\langle U(u,v),U(u,v)\rangle_u^F}
{\langle u,u\rangle_u^F \langle v,v\rangle_u^F-
{\langle u,v\rangle_u^F}\langle u,v\rangle_u^F},
\end{equation*}
where $U$ is a bilinear map from $\mathfrak{m}\times \mathfrak{m}$ to $\mathfrak{m}$ defined by
\begin{equation*}
\langle U(u,v),w\rangle_u^F=\frac{1}{2}(\langle[w,u]_\mathfrak{m},v\rangle_u^F
+\langle[w,v]_\mathfrak{m},u\rangle_u^F), \mbox{ for any }w\in\mathfrak{m},
\end{equation*}
here $[\cdot,\cdot]_\mathfrak{m}=\mathrm{pr}_\mathfrak{m}\circ[\cdot,\cdot]$ and $\mathrm{pr}_\mathfrak{m}$ is the projection
with respect to the given $\mathrm{Ad}(H)$-invariant decomposition.
\end{theorem}

The flag curvature formula in Theorem \ref{flag-curvature-formula-thm}
only deals with some special flags spanned by commutative pairs, but it is very convenient in this paper. In \cite{XDHH2014}, we have provided a proof of this theorem by the submersion technique. To make this work more self-contained, we quote here another shorter proof
by L. Huang, which can also be found in \cite{XDHH2014}.

In \cite{Huang2013}, L.~Huang has obtain a general flag curvature formula
of homogeneous Finsler spaces, using the technique of invariant frames.
To introduce his formula, we first define the {\it spray vector field} $\eta:\mathfrak{m}\backslash\{0\}\rightarrow\mathfrak{m}$
and the {\it connection operator} $N:(\mathfrak{m}\backslash\{0\})\times\mathfrak{m}\rightarrow\mathfrak{m}$.
For any $u\in\mathfrak{m}\backslash\{0\}$,
$\eta(u)$ is defined by
\begin{equation*}
	\langle\eta(u), w\rangle_u^F=\langle u,[w,u]_{\mathfrak{m}}\rangle_u^F, \quad \forall w\in\mathfrak{m},
\end{equation*}
and $N(u,\cdot)$ is a linear operator
on $\mathfrak{m}$ determined by
\begin{equation*}
	\begin{aligned}
	2\langle N(u,w_1),w_2\rangle_u^F=\langle [w_2,w_1]_{\mathfrak{m}}, u\rangle_u^F & + \langle [w_2,u]_{\mathfrak{m}}, w_1\rangle_u^F+\langle [w_1,u]_{\mathfrak{m}},w_2\rangle_u^F\\
	& - 2C^F_u(w_1,w_2,\eta(u)),\quad \forall w_1,w_2\in\mathfrak{m}.
	\end{aligned}	
\end{equation*}
Using these two notions, L.~Huang proved the following formula
for Riemann curvature $R_u: T_o(G/H)\rightarrow T_o(G/H)$,
\begin{equation}\label{6000}
	\langle R_u(w),w\rangle_u^F
	=\langle [[w,u]_{\mathfrak{h}},w],u\rangle_u^F
	+\langle \tilde{R}(u)w,w\rangle_u^F,\quad \forall w\in\mathfrak{m},
\end{equation}
where the linear operator
$\tilde{R}(u):\mathfrak{m}\rightarrow\mathfrak{m}$
is given by
$$\tilde{R}(u)w=D_{\eta(u)}N(u,w)-N(u,N(u,w))+
N(u,[u,w]_\mathfrak{m})-[u,N(u,w)]_\mathfrak{m},$$
here $D_{\eta(u)}N(u,w)$ is the derivative of $N(\cdot,w)$ at $u\in\mathfrak{m}\backslash\{0\}$ in the direction of $\eta(u)$. In particular, if $\eta(u)=0$,
then $D_{\eta(u)}N(u,w)=0$.

Now suppose $u\in\mathfrak{m}\backslash\{0\}$ satisfies $\langle[u,\mathfrak{m}],u\rangle_u^F=0$, i.e.,
$\eta(u)=0$. Then for any $v\in\mathfrak{m}$  commutative with $u$,
we have $N(u,v)=U(u,v)$. Thus
\begin{eqnarray*}
\langle R_u (v),v\rangle_u^F &=&
-\langle N(u,N(u,v)),v\rangle_u^F
-\langle [u,N(u,v)],v\rangle_u^F\\
&=& -\frac12(\langle [v,N(u,v)]_\mathfrak{m},u\rangle_u^F+
\langle[N(u,v),u]_\mathfrak{m},v\rangle_u^F)
+\langle [N(u,v),u],v\rangle_u^F\\
&=& \frac12(\langle[N(u,v),v],u\rangle_u^F+
\langle[N(u,v),u],v\rangle_u^F)\\
&=&\langle U(u,v),N(u,v)\rangle_u^F=\langle U(u,v),U(u,v)\rangle_u^F.
\end{eqnarray*}
From this the flag curvature formula in Theorem \ref{flag-curvature-formula-thm} follows immediately.

\subsection{The totally geodesic technique and the rank equality}

Assume that $(G/H,F)$ is a positively curved homogeneous Finsler space,
with a bi-invariant orthogonal decomposition $\mathfrak{g}=\mathfrak{h}+\mathfrak{m}$ for the compact Lie group $\mathfrak{g}$.

Let $\mathfrak{t}$ be a Cartan subalgebra of $\mathfrak{g}$ such that
$\mathfrak{t}\cap\mathfrak{h}$ is a Cartan subalgebra of $\mathfrak{h}$.
For simplicity, we  just call  $\mathfrak{t}$ a {\it fundamental Cartan subalgebra}.
Fix a subalgebra $\mathfrak{t}'$ of $\mathfrak{t}\cap\mathfrak{h}$,
and denote the identity component of $C_G(\mathfrak{t}')$ as $G'$. Let
$H'=G'\cap H$. Then $(G'/H',F|_{G'/H'})$ is a homogeneous
submanifold of $(G/H,F)$.

We first prove the following useful lemma.

\begin{lemma} \label{totally-geodesic-lemma}
Keep all the above notation. Then $(G'/G'\cap H, F|_{G'/H'})$
is totally geodesic in $(G/H,F)$. In particular, if $G/H$ admits positively
curved homogeneous Finsler metrics and $\dim G'/H'>1$, then $G'/H'$ also admits positively curved homogeneous Finsler metrics.
\end{lemma}

\begin{proof}
We will present two proofs of the theorem. The first proof uses Corollary II.5.7 of
\cite{Br}, which asserts that the set of  common fixed points of $T'$ is a disconnected
union of finite orbits of $N_G(T')=\{g\in G|g^{-1}T'g=T'\}$. Thus the connected component of $N_G(T')\cdot o$ containing $o=eH$, which coincides with $G'/H'$, is a totally geodesic submanifold of
$(G/H,F)$. Therefore, if  $(G/H,F)$ is positively curved and $\dim G'/H' >1$, then the homogeneous Finsler space $(G'/H',F|_{G'/H'})$
has positive flag curvature.

The second proof will be completed through a  direct calculation, using the geodesic spray formula for homogeneous Finsler spaces in \cite{DX2014-2}.

Note that the Lie algebra $\mathfrak{g}'$ of $G'$ is the centralizer
$C_{\mathfrak{g}}(\mathfrak{t}')$ of $\mathfrak{t}'$ in
$\mathfrak{g}$. Since $\mathfrak{t}'\subset\mathfrak{h}$, we also have the
decomposition
$\mathfrak{g}'=(\mathfrak{g}'\cap\mathfrak{h})+(\mathfrak{g}'\cap\mathfrak{m})$,
where $\mathfrak{g}'\cap\mathfrak{h}=C_\mathfrak{h}(\mathfrak{t}')$ is
the Lie algebra of $H'$.
Since the bi-invariant orthogonal complement ${\mathfrak{g}'}^{\perp}$ is equal to $[\mathfrak{t}',\mathfrak{g}]$,
we also have
$$[\mathfrak{t}',\mathfrak{g}]=[\mathfrak{t}',\mathfrak{h}]+
[\mathfrak{t}',\mathfrak{m}]=({\mathfrak{g}'}^{\perp}\cap\mathfrak{h})
+({\mathfrak{g}'}^{\perp}\cap\mathfrak{m}).$$
Let $v_1$, $\ldots$, $v_m$, $v_{m+1}$, $\ldots$, $v_{m+n}$ be an orthogonal
basis of $\mathfrak{m}$ with respect to the  bi-invariant inner product,
such that $v_i\in \mathfrak{g}'\cap\mathfrak{m}$, for $1\leq i\leq m$, and denote the Killing vector field on $M=G/H$ generated by $v_j$  as $X_j$, for $1\leq j\leq m+n$. Then the the restriction of $X_i$ to $M'$ is a  Killing vector fields of $(M',F|_{M'})$, $1\leq i\leq m$. It is easily seen that there is an neighborhood $U$ of the origin
$o$, such that  $y=y^iX_i$ defines a linear coordinate  system for  $y\in TU$.

We now consider the geodesic spray $\mathbf{G}(o,y)$ of $(M,F)$ at $o$ when $y$ lies in the linear span of  $v_1$, $\ldots$, $v_m$.

In \cite{XD2014-2}, we have proven that
\begin{equation}\label{geodesic-formula}
\mathbf{G}(o,y)=y^i\tilde{X}_i+ g^{il}c^k_{jl}g_{kh}y^h y^j\partial_{y^i},
\end{equation}
where $\tilde{X}_i$ is the tangent vector field on
$T(TM\backslash\{0\})$ naturally induced by $X_i$, and the coefficients $c^k_{ij}$ are defined by
$[v_i,v_j]_\mathfrak{m}=c^k_{ij}v_k$.
Since $[C_{\mathfrak{g}}(\mathfrak{t}'),[\mathfrak{t}',\mathfrak{g}]]
\subset[\mathfrak{t}',\mathfrak{g}]$, we have
$[\mathfrak{g}'\cap\mathfrak{m},[\mathfrak{t}',\mathfrak{g}]\cap\mathfrak{m}]_\mathfrak{m}
\subset[\mathfrak{t}',\mathfrak{g}]\cap\mathfrak{m}$, hence $c^k_{ij}=0$, for
$i\leq m$, $j>m$ and $k\leq m$. On the other hand, since $F$ is
$\mathrm{Ad}(H)$-invariant, by \cite{DH2004}, we have
$$\langle [h,v],w\rangle_y^{F}+\langle v,[h,w]\rangle_y^F=
-2C_u([h,y],v,w),\quad\forall h\in\mathfrak{h},v\in\mathfrak{g}'\cap\mathfrak{m}, w\in\mathfrak{m}.$$
In particular,  for $h\in\mathfrak{t}'$, we have $[h,v]=[h,y]=0$. Then we have
\begin{equation}\label{yk}
\langle \mathfrak{g}'\cap\mathfrak{m},[\mathfrak{t}',
\mathfrak{m}]\rangle_u^F=
\langle \mathfrak{g}'\cap\mathfrak{m},[\mathfrak{t}',
\mathfrak{g}]\cap\mathfrak{m}\rangle_u^F=0.
\end{equation}
We now suppose that   $y^k=0$ for any $k>m$. Then (\ref{yk}) implies that  $g^{ij}=g_{ij}=0$   for  $1\leq i\leq m<j\leq m+n$.
Hence in this case, the only nonzero terms in the right side of
(\ref{geodesic-formula}) are $y^i X_i$ with $i\leq m$, and
$g^{il}c^k_{jl}g_{kh}y^h y^j\partial_{y^i}$ with $i,j,k,h,l\leq m$. Consequently  in this case
$\mathbf{G}(o,y)$ is equal to the geodesic spray of $(G'/H',F|_{G'/H'})$ at $(o,y)$.
By the homogeneity,   the above assertion is  valid for any $g\in G'$. Therefore  $(G'/H',F|_{G'/H'})$ is totally geodesic in $(G/H,F)$.
\end{proof}

From the first proof, we see that the lemma is still valid with $T'$ changed to other subgroups in $H$. This will be convenient when it is difficult to calculate directly with  $G'$, but up
to equivalence, $G'$ and $H'$ contains a common product factor $T'$ which can be cancelled.

We now give an immediate application of Lemma \ref{totally-geodesic-lemma}.
Assume that $\mathfrak{t}'=\mathfrak{t}\cap\mathfrak{h}$,
and $F'=F|_{G'/H'}$ induces a left invariant Finsler metric $F''$ on the compact Lie
group $G''$ with $\mathrm{Lie}(G'')=  C_\mathfrak{g}(\mathfrak{t}\cap\mathfrak{h})\cap\mathfrak{m}$.
Then the above lemma implies that if $\dim G''>1$,
then $F''$ is positively curved. Thus by Theorem 5.1
of \cite{DH2013}, we have $G''=\mathrm{U}(1)$, $\mathrm{SU}(2)$ or $\mathrm{SO}(3)$. This proves the following rank equality, which is a special case of Theorem 5.2 in \cite{XDHH2014}.
\begin{corollary} \label{rank-equality-corollary}
Let $(G/H,F)$ be an odd dimensional positively curved
homogeneous Finsler space with compact $\mathfrak{g}=\mathrm{Lie}(G)$. Then
$\mathrm{rk}\mathfrak{g}=\mathrm{rk}\mathfrak{h}+1$.
\end{corollary}

\subsection{Some notation for  Lie algebras and root systems}

 We now set some notation for the relevant Lie algebras and root systems. Let $(G/H,F)$ be an odd dimensional positively curved homogeneous Finsler space with a bi-invariant orthogonal decomposition $\mathfrak{g}=\mathfrak{h}+\mathfrak{m}$ for the compact Lie algebra $\mathfrak{g}=\mathrm{Lie}(G)$.
The orthogonal projections to the $\mathfrak{h}$-factor
and $\mathfrak{m}$-factor are denoted as $\mathrm{pr}_\mathfrak{h}$
and $\mathrm{pr}_\mathfrak{m}$, respectively.

Fix a fundamental Cartan subalgebra $\mathfrak{t}$ of $\mathfrak{g}$ (i.e., $\mathfrak{t}\cap\mathfrak{h}$ is a Cartan subalgebra of $\mathfrak{h}$).
From now on,  root systems, root planes, etc, for $\mathfrak{g}$ will be taken with
respect to $\mathfrak{t}$, and those for $\mathfrak{h}$ will be taken  with respect to $\mathfrak{t}\cap\mathfrak{h}$. It is easy to see that $\mathfrak{t}$ is a splitting Cartan subalgebra, that is,
$$\mathfrak{t}=(\mathfrak{t}\cap\mathfrak{h})+(\mathfrak{t}\cap\mathfrak{m}).$$
By Corollary \ref{rank-equality-corollary}, we have
$\dim (\mathfrak{t}\cap\mathfrak{m})=1$.

The maximal torus of $G$ (resp. $H$) corresponding to $\mathfrak{t}$ (resp. $\mathfrak{t}\cap\mathfrak{h}$) will be denoted as $T$ (resp. $T_H$).
We now have the following decomposition of $\mathfrak{g}$ with respect to $\mathrm{Ad}(T)$-actions:
\begin{equation}
\mathfrak{g}=\mathfrak{t}+\sum_{\alpha\in\Delta_{\mathfrak{g}}}\mathfrak{g}_{\pm\alpha},
\end{equation}
 where $\Delta_\mathfrak{g}\subset\mathfrak{t}$ is the root system of $\mathfrak{g}$.
 For  $\alpha\in\Delta_\mathfrak{g}$,  $\mathfrak{g}_{\pm\alpha}$  is a two dimensional
irreducible representation of $\mathrm{Ad}(T)$-actions, called a {\it root plane}.
Through the  bi-invariant inner product, we will regard a root
as a vector in $\mathfrak{t}$ rather than a vector in $\mathfrak{t}^*$.
For the compact Lie algebra $\mathfrak{h}=\mathrm{Lie}(H)$, we have a similar
decomposition with respect to $\mathrm{Ad}(T_H)$-actions. On the other hand,  root planes of
$\mathfrak{h}$ are denoted as $\mathfrak{h}_{\pm\alpha'}$,  where $\alpha'
\in\mathfrak{t}\cap\mathfrak{h}$ are roots of $\mathfrak{h}$ in the root system $\Delta_\mathfrak{h}\subset\mathfrak{t}\cap\mathfrak{h}$.

There is another decomposition of $\mathfrak{g}$ with respect to the $\mathrm{Ad}(T_H)$-action, namely,
\begin{equation}\label{decomposition-compact-lie-alg-2}
\mathfrak{g}=\sum_{\alpha'\in\mathfrak{t}\cap\mathfrak{h}}
\hat{\mathfrak{g}}_{\pm\alpha'},
\end{equation}
where
\begin{equation*}
\hat{\mathfrak{g}}_{\pm\alpha'}=\sum_{\mathrm{pr}_\mathfrak{h}(\alpha)=\alpha'}
\mathfrak{g}_{\pm\alpha},\quad \mbox{if}\,\, \alpha'\neq 0,
\end{equation*}
 $\hat{\mathfrak{g}}_0=\mathfrak{t}+\mathfrak{g}_{\pm\alpha}$, if  there is a root $\alpha$ of $\mathfrak{g}$ contained in $\mathfrak{t}\cap\mathfrak{m}$,
 and $\hat{\mathfrak{g}}_0=\mathfrak{t}\cap\mathfrak{m}$ otherwise. This {\it $\mathrm{Ad}(T_H)$-invariant decomposition} is compatible with the bi-invariant orthogonal decomposition in the sense that
$$\hat{\mathfrak{g}}_{\pm\alpha'}=(\hat{\mathfrak{g}}_{\pm\alpha'}\cap\mathfrak{h})
+(\hat{\mathfrak{g}}_{\pm\alpha'}\cap\mathfrak{m}).$$
To be more precise, we have the following easy lemma, which will be repeatedly used in the sequel.
\begin{lemma} Let $\alpha'$ be a vector of $\mathfrak{t}\cap\mathfrak{h}$. Then we have the following:
\begin{description}
\item{\rm (1)}
if $\alpha'\in\Delta_\mathfrak{h}$, then we have
$\hat{\mathfrak{g}}_{\pm\alpha'}=(\hat{\mathfrak{g}}_{\pm\alpha'}
\cap\mathfrak{h})+
(\hat{\mathfrak{g}}_{\pm\alpha'}\cap\mathfrak{m}),$
where $\hat{\mathfrak{g}}_{\pm\alpha'}
\cap\mathfrak{h}=\mathfrak{h}_{\pm\alpha'}$;
\item{\rm (2)} if $\alpha'\notin\Delta_\mathfrak{h}$, then we have
$\mathfrak{g}_{\pm\alpha'}\subset\mathfrak{m}$. In particular, $\hat{\mathfrak{g}_0}\subset\mathfrak{m}$, and
$\mathfrak{g}_{\pm\alpha}\subset\mathfrak{m}$, if $\mathrm{pr}_\mathfrak{h}\alpha\notin\Delta_\mathfrak{h}$.
\end{description}
\end{lemma}

For the bracket relation between root planes, we have the following
well known formula:
\begin{equation}
[\mathfrak{g}_{\pm\alpha},\mathfrak{g}_{\pm\beta}]\subseteq
\mathfrak{g}_{\pm(\alpha+\beta)}+\mathfrak{g}_{\pm(\alpha-\beta)},
\end{equation}
where $\mathfrak{g}_{\pm\alpha}$ and $\mathfrak{g}_{\pm\beta}$ are different
root planes, i.e., $\alpha\neq\pm\beta$, and each term of the right side can
be $0$ when the corresponding vector is not a root of $\mathfrak{g}$. In fact,
this is just a special case of the following
\begin{lemma}Keep all the above notation. We have
\begin{description}\label{trick-lemma-0}
\item{\rm (1)} For any root $\alpha$ of $\mathfrak{g}$, $[\mathfrak{g}_{\pm\alpha},
\mathfrak{g}_{\pm\alpha}]=\mathbb{R}\alpha$.
\item{\rm (2)} Let $\alpha$ and $\beta$ be two linearly independent roots  of $\mathfrak{g}$. If none of the roots $\alpha\pm\beta$ is a root of $\mathfrak{g}$, then
$[\mathfrak{g}_{\pm\alpha},\mathfrak{g}_{\pm\beta}]=0$; if one of $\alpha\pm\beta$, say $\gamma$,  is a root, and the other is not,  then $[\mathfrak{g}_{\pm\alpha},\mathfrak{g}_{\pm\beta}]=
\mathfrak{g}_{\gamma}$; If both $\alpha\pm\beta$ are  roots of $\mathfrak{g}$, then $[\mathfrak{g}_{\pm\alpha},\mathfrak{g}_{\pm\beta}]$ is a cone in $\mathfrak{g}_{\pm(\alpha+\beta)}+\mathfrak{g}_{\pm(\alpha-\beta)}$.
\item{\rm (3)} In the second case of (2),  for any nonzero vector $v\in\mathfrak{g}_{\pm\alpha}$, the linear map $\mathrm{ad}(v)$
 is an isomorphism from $\mathfrak{g}_{\pm\beta}$  onto $[\mathfrak{g}_{\pm\alpha},\mathfrak{g}_{\pm\beta}]=
 \mathfrak{g}_{\pm\gamma}$.
\end{description}
\end{lemma}

The root systems of compact simple Lie algebras $A_n$-$G_2$ and the presentations of root planes for the classical cases $A_n$-$D_n$ are
listed in the appendix (Section \ref{appendix-section}).

\subsection{The three Cases and the reversibility assumption}

Keep all the above assumptions and notation.
In \cite{XD2014-2}, we established the general theme for our classification of positively curved
normal homogeneous Finsler spaces. The main idea can be applied to this paper. In particular,  we only need to consider  the following three cases:
\begin{description}
\item{\bf Case I.} Each root plane of $\mathfrak{h}$ is also a root plane of $\mathfrak{g}$.
\item{\bf Case II.} There are roots $\alpha$ and $\beta$ of $\mathfrak{g}$ of $\mathfrak{g}$ from different simple factors, such that $\mathrm{pr}_\mathfrak{h}(\alpha)=\mathrm{pr}_\mathfrak{h}(\beta)=\alpha'$
    is a root of $\mathfrak{h}$.
\item{\bf Case III.} There exists a linearly independent pair of roots $\alpha$ and $\beta$ of $\mathfrak{g}$ from the same simple factor, such that $\mathrm{pr}_\mathfrak{h}(\alpha)=\mathrm{pr}_\mathfrak{h}(\beta)=\alpha'$
    is a root of $\mathfrak{h}$.
\end{description}

In the following sections, we will restrict our discussion to
reversible Finsler metrics (i.e. $F(x,y)=F(x,-y)$ for any $y\in T_x (G/H)$). The reason for adding this condition for $F$ will be explained in the next
subsection.

It turns out that with the reversibility assumption for $F$, Case II is the easiest. Case III contains a lot of case-by-case discussions. But in this case we can use the root $\alpha'$ of
$\mathfrak{h}$ to settle the problem.
Case I turns out to be very difficult, and we can only get some partial result for this case.

Adding the reversibility assumption will not lose too much  generality, and it provides an alternative certification that the  classification result in \cite{BB76} is correct.

\subsection{The key lemmas for reversible metrics}
\label{subsection-key-lemmas}

From now on, we will assume that $(G/H,F)$ is an odd dimensional positively curved reversible homogeneous Finsler space, with a bi-invariant orthogonal decomposition
$\mathfrak{g}=\mathfrak{h}+\mathfrak{m}$ for the compact Lie algebra $\mathfrak{g}=\mathrm{Lie}(G)$, and a fundamental Cartan subalgebra $\mathfrak{t}$.
Keep all the relevant notation as before.

In the following three lemmas we will present some results on
the $g_u^F$-orthogonal (i.e. with respect to the inner product $\langle\cdot,\cdot\rangle_u^F$) decomposition of $\mathfrak{m}$. These lemmas are crucial for our later discussions.

\begin{lemma} \label{lemma-3-7}
Keep the above assumptions and notation.
\begin{description}
\item{\rm (1)}\quad Let $u$ be a nonzero
vector in $\hat{\mathfrak{g}}_{0}\subset\mathfrak{m}$. Then $\mathfrak{m}$ has a $g_u^F$-orthogonal decomposition as the sum of all
$\hat{\mathfrak{m}}_{\pm\alpha'}=\hat{\mathfrak{g}}_{\pm\alpha'}\cap
\mathfrak{m}$, $\alpha'\in\mathfrak{t}\cap\mathfrak{h}$. In particular, $\hat{\mathfrak{m}}_0=\hat{\mathfrak{g}}_0$.
\item{\rm (2)}\quad If $\dim\hat{\mathfrak{g}}_{0}=3$, then there is a  fundamental Cartan subalgebra $\mathfrak{t}$, such that for
    any nonzero vector $u\in\mathfrak{t}\cap\mathfrak{m}$, we have $\langle\mathfrak{t}\cap\mathfrak{m},\mathfrak{g}_{\pm\alpha}\rangle_u^F=0$, where $\alpha$ is the root in $\mathfrak{t}\cap\mathfrak{m}$.
\end{description}
\end{lemma}

\begin{proof} (1) Let $T_H$ be the torus in $H$ with $\mathrm{Lie}(T_H)=\mathfrak{t}\cap\mathfrak{h}$.
Since both $F$ and $u\in\hat{\mathfrak{g}}_0$ are $\mathrm{Ad}(T_H)$-invariant, the inner product
$\langle\cdot,\cdot\rangle_u^F$ is also $\mathrm{Ad}(T_H)$-invariant.
The summands  given in the decomposition correspond to different irreducible
representations of $T_H$, thus it is a $g_u^F$-orthogonal decomposition.

(2) Choose the $F$-unit vector $u\in\hat{\mathfrak{g}}_0$ such that
$||u||_{\mathrm{bi}}$ reaches the maximum among all $F$-unit vectors in
$\hat{\mathfrak{g}}_0$.
Then $\mathfrak{t}_0=\mathfrak{t}\cap\mathfrak{h}+\mathbb{R}u$ is also a fundamental Cartan subalgebra of $\mathfrak{g}$. Notice that for  $\alpha'\in\mathfrak{t}\cap\mathfrak{h}$, the subspace $\hat{\mathfrak{g}}_{\pm\alpha'}$  does not change when $\mathfrak{t}$ is replaced  with $\mathfrak{t}_0$. The bi-invariant orthogonal
complement $u^\perp\cap\hat{\mathfrak{g}}_0$ of $u$ in $\hat{\mathfrak{g}}_0$ is a root plane $\mathfrak{g}_{\pm\alpha}$ for $\mathfrak{t}_0$. Then our assumption on $u$ implies that
$$\langle\mathfrak{t}_0\cap\mathfrak{m},\mathfrak{g}_{\pm\alpha}\rangle_u^F
=\langle \mathbb{R}u,u^\perp\cap\mathfrak{g}_0\rangle_u^F=0.$$
This completes the proof of the lemma.\end{proof}

\begin{lemma}\label{lemma-3-8}
Keep the above assumptions and notation. Let  $u\in\mathfrak{m}$ be a  nonzero vector in
a root plane $\hat{\mathfrak{m}}_{\pm\alpha'}$ with $\alpha'\neq 0$.
Denote the bi-invariant orthogonal complement of $\alpha$ in $\mathfrak{t}\cap\mathfrak{h}$ as $\mathfrak{t}'$, and the bi-invariant orthogonal projection to $\mathfrak{t}'$ as $\mathrm{pr}_{\mathfrak{t}'}$.
Then $\mathfrak{m}$ can be $g_u^F$-orthogonally decomposed as the sum of
\begin{eqnarray*}
\hat{\hat{\mathfrak{m}}}_{\pm\gamma''}&=&
(\sum_{\mathrm{pr}_{\mathfrak{t}'}(\gamma)=\gamma''}\mathfrak{g}_{\pm\gamma})\
\cap\mathfrak{m}
=\sum_{\mathrm{pr}_{\mathfrak{t}'}(\gamma')=\gamma''}
(\hat{\mathfrak{g}}_{\pm\gamma'}\cap\mathfrak{m})\\
&=&(\sum_{\gamma\in\tau+\mathbb{R}\alpha+\mathfrak{t}\cap\mathfrak{m}}
\mathfrak{g}_{\pm\gamma})\cap\mathfrak{m},
\end{eqnarray*}
where $\tau$ is a root of $\mathfrak{g}$ with $\mathrm{pr}_{\mathfrak{t}'}(\tau)=\gamma''$.
In particular, $\hat{\hat{\mathfrak{m}}}_0=
(\mathop\sum\limits_{\gamma\in\mathbb{R}\alpha+\mathfrak{t}\cap\mathfrak{m}}
\mathfrak{g}_{\pm\gamma})\cap\mathfrak{m}$.
\end{lemma}

\begin{proof} Let $T'$ be the torus in $H$ with $\mathrm{Lie}(T')=\mathfrak{t}'$.
Since both $F$ and $u$ are $\mathrm{Ad}(T')$-invariant,  the inner product
$\langle\cdot,\cdot\rangle_u^F$ on $\mathfrak{m}$ is also $\mathrm{Ad}(T')$-invariant. The summands   given in the  decomposition correspond to different irreducible representations of $T'$, thus it is an orthogonal
decomposition with respect to $\langle\cdot,\cdot\rangle_u^F$.
\end{proof}

The following lemma does not hold in general without the reversibility assumption.

\begin{lemma} \label{lemma-3-6}
Keep the above assumptions and notation.
Then for any nonzero vector
$u\in\hat{\mathfrak{m}}_{\pm\alpha'}=
\hat{\mathfrak{g}}_{\pm\alpha'}\cap\mathfrak{m}$ with $\alpha'\neq 0$, and any $\beta'\in\mathfrak{t}\cap\mathfrak{h}$ which is not an even multiple of $\alpha'$,  we have
$$\langle\hat{\mathfrak{m}}_{\pm\beta'},\hat{\mathfrak{g}}_0
\rangle_u^F=0.$$
In particular, we have
$$\langle \hat{\mathfrak{m}}_{\pm\alpha'},
\hat{\mathfrak{g}}_{0}\rangle_u^F=0.$$
\end{lemma}

\begin{proof} Without losing generality, we can assume that  $\hat{\mathfrak{m}}_{\pm\beta'}\ne 0$. Then $\dim\hat{\mathfrak{m}}_{\pm\beta'}=2k>0$ is even. Hence  there exists
 an element $g$ in the maximal torus $T_H$ of $H$, and a bi-invariant orthonormal basis $\{u_1,v_1,u_2,v_2,\ldots,u_k,v_k\}$ of $\hat{\mathfrak{g}}_{\pm\beta'}\cap\mathfrak{m}$ such
that
$\mathrm{Ad}(g)|_{\hat{\mathfrak{g}}_{\pm\alpha'}}=-\mathrm{Id}$,
$\mathrm{Ad}(g)|_{\hat{\mathfrak{g}}_{0}\cap\mathfrak{m}}=\mathrm{Id}$,
and for each $i$, $\mathrm{Ad}(g)|_{\mathbb{R}u_i+\mathbb{R}v_i}$ is the anticlockwise rotation $R(\theta)$ with  angle $\theta\in (0,2\pi)$.

Since $F$ is $\mathrm{Ad}(g)$-invariant, for any
$w_1\in\mathbb{R}u_i+\mathbb{R}v_i$ and
$w_2\in\in\hat{\mathfrak{g}}_{0}\cap\mathfrak{m}$, we have
$$\langle w_1,w_2\rangle_u^F=
\langle\mathrm{Ad}(g)w_1,\mathrm{Ad}(g)w_2\rangle_{\mathrm{Ad}(g)u}^F
=\langle R(\theta)w_1,w_2\rangle_{-u}^F=\langle R(\theta)w_1,w_2\rangle_u^F.$$
Repeating this procedure, we get
$\langle w_1,w_2\rangle_u^F=\langle R(n\theta)w_1,w_2\rangle_u^F$ for each $n\in\mathbb{N}$. So
$$\langle w_1,w_2\rangle_u^F=\lim_{n\rightarrow\infty}
\langle\frac1n(R(\theta)w_1+\cdot+R(n\theta)w_1),w_2\rangle_u^F=0.$$
Now the above argument holds for any  $i$ between $1$ to $k$. This   proves the lemma.
\end{proof}

The following two lemmas will be repeatedly used  in our later discussion.

\begin{lemma}\label{key-lemma-1} Let $F$ be a positively curved homogeneous Finsler metric on the odd dimensional coset space $G/H$. Keep all the relevant notation as before.
If $\alpha$ is a root of $\mathfrak{g}$
contained in $\mathfrak{t}\cap\mathfrak{h}$, and it is the only root
of $\mathfrak{g}$ contained in $\alpha+(\mathfrak{t}\cap\mathfrak{m})$,
then it is a root of $\mathfrak{h}$ and we have
$\mathfrak{h}_{\pm\alpha}=\hat{\mathfrak{g}}_{\pm\alpha}
=\mathfrak{g}_{\pm\alpha}$.
\end{lemma}

\begin{proof} We only need to prove that $\alpha$ is a root of $\mathfrak{h}$.
The other statement follows easily.

Assume conversely that $\alpha$ is not a root
of $\mathfrak{h}$. Then
$\mathfrak{g}_{\pm\alpha}=\hat{\mathfrak{g}}_{\pm\alpha}$ is contained in
$\mathfrak{m}$. By (2) of Lemma \ref{lemma-3-7}, if $\dim\hat{\mathfrak{g}}_0=3$, then there exists a  fundamental Cartan subalgebra $\mathfrak{t}$ and a nonzero $u$ in $\mathfrak{t}\cap\mathfrak{m}$, such that
\begin{equation}\label{2100}
\langle u^\perp\cap\hat{\mathfrak{g}}_0,u\rangle_u^F=0,
\end{equation}
 where $u^\perp\cap\hat{\mathfrak{g}}_0$ is the bi-invariant orthogonal complement of $u$ in $\hat{\mathfrak{g}}_0$.
Let $v$ be
a nonzero vector in $\mathfrak{g}_{\pm\alpha}$. Since $\alpha\in\mathfrak{t}\cap\mathfrak{h}$, it is easy to see that $u$ and $v$ are
 linearly independent and commutative.

Let $\alpha'=\mathrm{pr}_\mathfrak{h}(\alpha)$.
Then a direct calculation shows that
\begin{equation*}
[u,\mathfrak{m}]_{\mathfrak{m}}\subset u^\perp\cap\hat{\mathfrak{g}}_0
+\sum_{\gamma'\neq\alpha'}
\hat{\mathfrak{g}}_{\pm\gamma'}.
\end{equation*}
Thus by (\ref{2100}) and (1) of Lemma \ref{lemma-3-7},
we have
\begin{equation}\label{2101}
\langle[u,\mathfrak{m}]_\mathfrak{m},u\rangle_u^F=
\langle[u,\mathfrak{m}]_\mathfrak{m},v\rangle_u^F=0.
\end{equation}
On the other hand, a direct calculation also shows that
\begin{equation*}
{[}v,\mathfrak{m}{]}_\mathfrak{m}\subset\sum_{\gamma'\neq 0}
\hat{\mathfrak{g}}_{\pm\gamma'}.
\end{equation*}
Hence by (1) of Lemma \ref{lemma-3-7}, we have
\begin{equation}\label{2102}
\langle[v,\mathfrak{m}]_\mathfrak{m},u\rangle_u^F=0.
\end{equation}

Taking the summation of (\ref{2101}) and (\ref{2102}), we get $U(u,v)=0$. Hence by Theorem \ref{flag-curvature-formula-thm}, we have
$K^F(o,u,u\wedge v)=0$.
This is a contradiction.
\end{proof}

\begin{lemma} \label{key-lemma-2}
Let $F$ be a reversible positively curved homogeneous Finsler metric on an odd dimensional coset space $G/H$. Keep all the relevant notation as before. Then
there does not exist a pair of  linearly independent roots $\alpha$ and
$\beta$ of $\mathfrak{g}$ such that the following (1)-(4) hold  simultaneously:
\begin{description}
\item{\rm (1)}\quad Neither $\alpha$ nor $\beta$ is a root of $\mathfrak{h}$;
\item{\rm (2)}\quad None of $\alpha\pm\beta$ is a root of $\mathfrak{g}$;
\item{\rm (3)}\quad $\pm\alpha$ are the only roots of $\mathfrak{g}$ in $\mathbb{R}\alpha+\mathfrak{t}\cap\mathfrak{m}$;
\item{\rm (4)}\quad $\pm\beta$ are the only roots of $\mathfrak{g}$ in
$\mathbb{R}\alpha\pm\beta+\mathfrak{t}\cap\mathfrak{m}$.
\end{description}
\end{lemma}

\begin{proof}
Assume conversely that there are roots $\alpha$ and $\beta$ of $\mathfrak{g}$
satisfying (1)-(4) of the lemma. Denote $\alpha'=\mathrm{pr}_{\mathfrak{h}}(\alpha)$ and $\beta'=\mathrm{pr}_{\mathfrak{h}}(\beta)$.
Then $\mathfrak{g}_{\pm\alpha}$ must be contained
in $\mathfrak{m}$, otherwise by (3) of the lemma,
$\mathfrak{g}_{\pm\alpha}=\hat{\mathfrak{g}}_{\pm\alpha'}$ is a root plane in $\mathfrak{h}$, hence
$\alpha\subset[\mathfrak{g}_{\pm\alpha},\mathfrak{g}_{\pm\alpha}]
\subset\mathfrak{h}$ is a root of $\mathfrak{h}$, which
is a contradiction to (1). Similarly, by (4) of the lemma, $\mathfrak{g}_{\pm\beta}=\hat{\mathfrak{g}}_{\pm\beta'}$ is also contained in $\mathfrak{m}$.

First we consider the case that $\alpha'\neq 0$, i.e., $\alpha$ is not contained by
$\mathfrak{t}\cap\mathfrak{m}$. Let $u$ and $v$ be any nonzero vectors in $\mathfrak{g}_{\pm\alpha}$ and $\mathfrak{g}_{\pm\beta}$ respectively. By (1) of the lemma and the above argument, they must be  linearly independent and commutative.

Let $u'$ be another nonzero vector in $\mathfrak{g}_{\pm\alpha}$ such that
$\langle u,u'\rangle_{\mathrm{bi}}$=0. By the $\mathrm{Ad}(T_H)$-invariance of
$F|_{\mathfrak{g}_{\pm\alpha}}$,  it coincides with the restriction of
the bi-invariant inner product up to a scalar. So we have
\begin{equation}\label{2109}
\langle u^\perp\cap\mathfrak{g}_{\pm\alpha},u\rangle_u^F=\langle\mathbb{R}u',u\rangle_u^F=0,
\end{equation}
 where $u^\perp\cap\mathfrak{g}_{\pm\alpha}=\mathbb{R}u'$ is the bi-invariant orthogonal complement of $u$ in $\mathfrak{g}_{\pm\alpha}$.

Let $\mathfrak{t}'$ be the bi-invariant orthogonal complement of $\alpha$ in
$\mathfrak{h}$, and $\mathrm{pr}_{\mathfrak{t}'}$ be the orthogonal projection to
$\mathfrak{t}'$ with respect to the bi-invariant inner product. By Lemma \ref{lemma-3-8},  $\mathfrak{m}$ can be $g_u^F$-orthogonally decomposed as the sum of
$$\hat{\hat{\mathfrak{m}}}_{\pm\gamma''}=
(\sum_{\mathrm{pr}_{\mathfrak{t}'}(\gamma)=\gamma''}\mathfrak{g}_{\gamma})
\cap\mathfrak{m}$$ for all different $\{\pm\gamma''\}\subset\mathfrak{t}'$. In particular,  (3) and (4) of the lemma indicates that
\begin{equation}\label{2110}
\hat{\mathfrak{g}}_0=\mathfrak{t}\cap\mathfrak{m},\mbox{ }
\hat{\hat{\mathfrak{m}}}_0=\mathfrak{t}\cap\mathfrak{m}+\mathfrak{g}_{\pm\alpha},
\mbox{ and }
\hat{\hat{\mathfrak{m}}}_{\pm\beta''}=
\mathfrak{g}_{\pm\beta},
\end{equation}
where $\beta''=\mathrm{pr}_{\mathfrak{t}'}(\beta)$.

Now   (1), (2) of the lemma and a direct calculation implies that
$$[u,\mathfrak{m}]\subset\mathfrak{t}\cap\mathfrak{m}+u^\perp\cap\mathfrak{g}_{\pm\alpha}
+\sum_{\gamma''\neq 0,\gamma''\neq\pm\beta'',}
\hat{\hat{\mathfrak{m}}}_{\pm\gamma''}.$$
So by Lemma \ref{lemma-3-6}, Lemma \ref{lemma-3-8} and (\ref{2109}),
we have
\begin{equation}\label{2120}
\langle[u,\mathfrak{m}]_\mathfrak{m},u\rangle_u^F=
\langle[u,\mathfrak{m}]_\mathfrak{m},v\rangle_u^F=0.
\end{equation}
On the other hand, a direct calculation also shows that
$$[v,\mathfrak{m}]_\mathfrak{m}\subset \hat{\mathfrak{g}}_0 +
\sum_{\gamma''\neq 0}\hat{\hat{\mathfrak{m}}}_{\pm\gamma''}.$$
Thus by Lemma \ref{lemma-3-6} and Lemma \ref{lemma-3-8},
we have
\begin{equation}\label{2121}
\langle[v,\mathfrak{m}]_\mathfrak{m},u\rangle_u^F=0.
\end{equation}
Taking the  summation of  (\ref{2120}) and (\ref{2121}), we get $U(u,v)=0$. Hence by Theorem \ref{flag-curvature-formula-thm}, we have  $K^F(o,u,u\wedge v)=0$.
This is a contradiction.
\end{proof}

Notice that Lemmas \ref{lemma-3-7},  \ref{lemma-3-8} and  \ref{key-lemma-1} does not require $F$ to be reversible.
For most cases in later discussions, the key lemmas will be  enough
 to deduce our classification. But in some cases (Subsection 5.5 for example),
we need to use Theorem \ref{flag-curvature-formula-thm} to deduce some more delicate results
to complete the proofs.

\section{Case III: the general reduction and the classical groups}
In this section, we consider the Case III for classical groups.
\subsection{The general reduction}

Assume that $(G/H,F)$ is an odd dimensional positively curved reversible homogeneous Finsler space in Case III,
i.e., with respect to a bi-invariant decomposition $\mathfrak{g}=\mathfrak{h}+\mathfrak{m}$ for the compact Lie algebra $\mathfrak{g}=\mathrm{Lie}(G)$, and a fundamental Cartan subalgebra $\mathfrak{t}$, there exists a pair of roots $\alpha$ and $\beta$ of $\mathfrak{g}$ from the same simple factor, with $\alpha\neq\pm\beta$, such that
$\mathrm{pr}_\mathfrak{h}(\alpha)=\mathrm{pr}_\mathfrak{h}(\beta)=\alpha'$ is a root of
$\mathfrak{h}$. Obviously, in this case $\mathfrak{t}\cap\mathfrak{m}$ is
spanned by $\alpha-\beta$.

We first prove the following lemma.
\begin{lemma}\label{lemma-4-1}
Let $(G/H,F)$ be an odd dimensional positively curved reversible
  homogeneous Finsler space in Case III. Keep all the relevant notationa. Then $(G/H,F)$ is equivalent to a positively curved reversible homogeneous Finsler space $(G'/H',F')$ in which $G'$ is a compact simple Lie group.
\end{lemma}

\begin{proof}
Suppose $\mathfrak{g}$ has a  direct sum decomposition  as
$$\mathfrak{g}=\mathfrak{g}_0\oplus\mathfrak{g}_1\cdots\oplus\mathfrak{g}_n,$$
where $\mathfrak{g}_0$ is an abelian subalgebra, and for  $i>0$, $\mathfrak{g}_i$ is a simple ideal of $\mathfrak{g}$.
Let $\alpha$ and $\beta$ be two roots of $\mathfrak{g}_1$.
Then obviously the abelian factor $\mathfrak{g}_0$ is contained in $\mathfrak{h}$.

 Let $\gamma$ be a root of $\mathfrak{g}_i$ with $i>1$. Then  $\gamma$
is the only root contained in $\gamma+\mathfrak{t}\cap\mathfrak{m}$. Thus
by Lemma \ref{key-lemma-1}, $\gamma$ is a root of $\mathfrak{h}$ and
$\mathfrak{g}_{\pm\gamma}=\mathfrak{h}_{\pm\gamma}$ is contained in $\mathfrak{h}$.  Since the simple factor $\mathfrak{g}_i$, $i>1$
is algebraically generated by its root planes, we have $\mathfrak{g}_i\subset\mathfrak{h}$ for $i>1$. Let $G'/H'$ be the homogeneous
space corresponding to the pair $(\mathfrak{g}_1,\mathfrak{h}_1)$. Then $G'/H'$ admits a homogeneous Finsler metric $F'$ naturally induced by $F$, such that
$(G/H,F)$ is equivalent to $(G'/H',F')$. This completes the proof of the lemma.
\end{proof}

 Since Lemma \ref{key-lemma-1} holds without  the reversible assumption, Lemma \ref{lemma-4-1} is also valid for non-reversible metrics.

 In the following we will start a  case by case consideration of  the compact simple Lie algebras. However, there are some common situations which can be uniformly dealt with. We summarize them as the following lemma.

\begin{lemma}\label{lemma-999}
Let $(G/H,F)$ be an odd dimensional positively curved reversible homogeneous Finsler space in Case III, with  compact simple Lie algebra $\mathfrak{g}=\mathrm{Lie}(G)$. Then for any two different roots $\alpha$ and $\beta$ such that $\mathrm{pr}_\mathfrak{h}(\alpha)=\mathrm{pr}_{\mathfrak{h}}(\beta)=\alpha'$ is a root of $\mathfrak{h}$, the angle between $\alpha$ and $\beta$ can not be
$\frac{\pi}{3}$ or $\frac{2 \pi}{3}$.
\end{lemma}

\begin{proof}
First we assume that $\mathfrak{g}\neq \mathfrak{g}_2$ and prove that the angle between
$\alpha$ and $\beta$ can not be $\frac \pi 3$.
Assume conversely that the angle between $\alpha$ and $\beta$ is $\frac \pi 3$.
Let $\mathfrak{t}'=\alpha'^\perp\cap\mathfrak{t}\cap\mathfrak{h}=
(\mathbb{R}\alpha+\mathbb{R}\beta)^\perp\cap\mathfrak{t}
$ be the bi-invariant orthogonal
complement of $\alpha'$ in $\mathfrak{t}\cap\mathfrak{h}$, and $T'$ be the corresponding torus in $H$. Notice that there is a decomposition  $\mathrm{Lie}(C_G(T'))=\mathfrak{t}'\oplus A_2$, such that $\alpha$
and $\beta$ are roots of the $A_2$-factor.
By Lemma \ref{totally-geodesic-lemma}, there is a positively curved homogeneous Finsler space $(G''/H'',F'')$, where $\mathfrak{g}''=\mathrm{Lie}(G'')=\mathfrak{su}(3)$, and
$\mathfrak{h}''=\mathrm{Lie}(H'')=A_1$ is
linearly spanned by
$$
w_1=\sqrt{-1}\left(
    \begin{array}{ccc}
      -2 & 0 & 0 \\
      0 & 1 & 0 \\
      0 & 0 & 1 \\
    \end{array}
  \right),
w_2=\sqrt{-1}\left(
             \begin{array}{ccc}
               0 & \bar{a} & \bar{b} \\
               a & 0 & 0 \\
               b & 0 & 0 \\
             \end{array}
           \right),
\mbox{ and }$$
$$w_3=\frac13[w_1,w_2]=\left(
          \begin{array}{ccc}
            0 & \bar{a} & \bar{b} \\
            -a & 0 & 0 \\
            -b & 0 & 0 \\
          \end{array}
        \right),
$$
where $a,b\in\mathbb{C}$ and $(a,b)\neq (0,0)$.
But then $[w_3,w_1]$ is not contained in $\mathfrak{h}''$. This is a contradiction.

Now we prove that the angle between $\alpha$ and $\beta$ can not be $\frac{2\pi}{3}$.
Assume conversely that it is $\frac{2\pi}{3}$. Then  $\alpha'=\frac12(\alpha+\beta)$
is a root of $\mathfrak{h}$. But then $\gamma=2\alpha'=\alpha+\beta$ is a root
of $\mathfrak{g}$ contained in $\mathfrak{t}\cap\mathfrak{h}$, and it is
the only root contained in $\gamma+(\mathfrak{t}\cap\mathfrak{m})$. So by
Lemma \ref{key-lemma-1}, $\gamma=2\alpha'$ is also a root of $\mathfrak{h}$.
This is a contradiction.

Finally, we assume that $\mathfrak{g}=\mathfrak{g}_2$ and prove that the angle between $\alpha$ and $\beta$ can not be $\frac \pi 3$. If $\alpha$ and $\beta$ are short roots, then they
can be replaced with two long roots with  angle $\frac{2\pi}{3}$, which has already been proven to be impossible.
If $\alpha$ and $\beta$ are long roots, then $\alpha'=\frac12
(\alpha+\beta)$ is a root of $\mathfrak{h}$. By Lemma \ref{key-lemma-1} and a
similar argument as above, the short root $\gamma=\frac13(\alpha+\beta)=\frac23\alpha'$ is also a root of $\mathfrak{h}$. This is a contradiction.
\end{proof}

Now we start the case by case discussion. Notice that in the following, we always assume that the relevant coset space has been endowed with an invariant reversible Finsler metric with positive flag curvature. If a contradiction arises, then we can conclude that the coset space cannot be positively curved
in the reversible homogeneous sense.
In each case, we use the standard
presentation of the root systems (see Section \ref{appendix-section}), and divide the discussion into subcases with
respect to the rank of $G$, the long/short roots choices of $\alpha$ and $\beta$ and the angle between $\alpha$ and
$\beta$. Using the Weyl group actions and more outer automorphisms for $D_n$ and
$E_6$, the subcases can be reduced to the following.

\subsection{The case $\mathfrak{g}=A_n$}

We only need to consider the following subcases.

{\bf Subcase 1.}\quad $n=3$, and $\alpha=e_1-e_4$, $\beta=e_3-e_2$.

 In this case, we have $\mathfrak{t}\cap\mathfrak{m}=\mathbb{R}(e_1+e_2-e_3-e_4)$ and it is easy to see that  $\alpha'=\frac12(e_1-e_2+e_3-e_4)$ is a root of $\mathfrak{h}$.
By Lemma \ref{key-lemma-1}, $e_1-e_2$ and $e_3-e_4$ are roots of $\mathfrak{h}$. Notice that $\hat{\mathfrak{g}}_{\pm(e_1-e_2)}=
\mathfrak{g}_{\pm(e_1-e_2)}$ is a root plane of $\mathfrak{h}$. Let $\beta'=\frac12(-e_1+e_2+e_3-e_4)\in\mathfrak{t}\cap\mathfrak{h}$. Then
any non zero $u\in\mathfrak{g}_{\pm(e_1-e_2)}\subset\mathfrak{h}$ defines a linear isomorphism
\begin{equation}\label{0008}
\mathrm{ad}(u):\hat{\mathfrak{g}}_{\pm\alpha'}=\mathfrak{g}_{\pm(e_1-e_4)}+\mathfrak{g}_{\pm(e_2-e_3)}
\to\hat{\mathfrak{g}}_{\pm\beta'}=\mathfrak{g}_{\pm(e_2-e_4)}+\mathfrak{g}_{\pm(e_1-e_3)}.
\end{equation}
Since $u\in\mathfrak{h}$,  $\mathrm{ad}(u)$ preserves the
bi-invariant orthogonal decomposition. So
$\beta'=\frac12(-e_1+e_2+e_3-e_4)$ is also a root of $\mathfrak{h}$. Now we prove that
$\mathfrak{h}=B_2$ and its root system is
$$\{\pm(e_1-e_2),\pm(e_3-e_4), \pm\alpha',\pm\beta'\}.$$

Now we prove that up to conjugation, $\mathfrak{h}$ is uniquely determined.
By (\ref{0008}), it is easy to see that $\mathfrak{h}$ is uniquely determined by $\mathfrak{h}_{\pm\alpha'}$.
Let $\mathfrak{g}'$ be the subalgebra of $\mathfrak{g}$ isomorphic to $A_1\oplus A_1$, defined by
$$\mathfrak{g}'=\mathbb{R}\alpha+\mathbb{R}\beta+
\mathfrak{g}_{\pm(e_1-e_4)}+\mathfrak{g}_{\pm(e_2-e_3)},$$
and let $\mathfrak{h}'$ be the subalgebra of $\mathfrak{g}'$ defined by $\mathfrak{h}'=\mathbb{R}\alpha'+\mathfrak{h}_{\pm\alpha'}$. Suppose
$\mathfrak{t}'=\mathfrak{t}\cap\mathfrak{g}'=\mathbb{R}\alpha+\mathbb{R}\beta$ is a fundamental Cartan subalgebra of $\mathfrak{g}'$.
Then we also have the induced bi-invariant orthogonal decomposition
$\mathfrak{g}'=\mathfrak{h}'+\mathfrak{m}'$, such that $\mathfrak{m}'=\mathfrak{m}\cap\mathfrak{g}'$ and
$\mathfrak{t}'\cap\mathfrak{m}'=\mathfrak{t}\cap\mathfrak{m}$. Notice also that
$\mathfrak{h}'$ can not have nonzero intersection with either of the two
simple factors of $\mathfrak{g}'$, otherwise,  by $\mathrm{Ad}(\exp\mathfrak{h}')$-actions, the whole subalgebra
 $\mathfrak{h}'$ coincides with that factor, which is a contradiction with the fact that $\mathfrak{h}'$ is diagonal in $\mathfrak{g}'$.

The following lemma will be useful.

\begin{lemma} \label{diagonal-A-1-conjugation-lemma}
Let $\mathfrak{g}'=\mathfrak{g}_1\oplus\mathfrak{g}_2=A_1\oplus A_1$ be endowed with a bi-invariant inner product. Assume that $\mathfrak{t}'$ is
a Cartan subalgebra, and $\mathfrak{h}'$ and $\mathfrak{h}''$ are subalgebras of $\mathfrak{g}'$
isomorphic to $A_1$ satisfying the following conditions:
\begin{description}
\item{\rm (1)}\quad $\mathfrak{h}'\cap\mathfrak{t}'=\mathfrak{h}''\cap\mathfrak{t}'$ is one
    dimensional.
\item{\rm (2)}\quad
    $\mathfrak{h}'\cap\mathfrak{g}_i=
     \mathfrak{h}''\cap\mathfrak{g}_i=0$,
    $i=1$, $2$.
\item{\rm (3)}\quad
    $\mathfrak{h}'\cap(\mathfrak{h}'\cap\mathfrak{t}')^\perp
    \subset\mathfrak{t}'^\perp$, and
    $\mathfrak{h}''\cap(\mathfrak{h}''\cap\mathfrak{t}')^\perp\subset
    \mathfrak{t}'^\perp$, where the orthogonal complements are taken
    with respect to the chosen bi-invariant inner product on $\mathfrak{g}$.
\end{description}
Then there is an $\mathrm{Ad}(\exp\mathfrak{t}')$-action which maps
$\mathfrak{h}'$ onto $\mathfrak{h}''$.
\end{lemma}

\begin{proof}
We first give a definition. For a compact Lie algebra of type $A_1$ endowed with a bi-invariant inner product, we call an orthogonal basis $\{u_1,u_2,u_3\}$ standard, if all the basis vectors have the same length, and
they satisfy the condition  $[u_i,u_j]=u_k$ for $(i,j,k)=(1,2,3)$, $(2,3,1)$
or $(3,1,2)$. The length $c$ of each $u_i$ is a constant which only depends
on the scale of the bi-invariant inner product. The bracket $u'_3=[u'_1,u'_2]$ of any two
orthogonal vectors with length $c$ is also a vector with  length $c$, and
$\{u'_1,u'_2,u'_3\}$ is a standard basis as well.

Now we go back to the proof. Let $c_1$ and $c_2$ be the length of
the standard basis vectors for $\mathfrak{g}_1$ and $\mathfrak{g}_2$,
respectively. Then we can choose standard bases $\{u_1,u_2,u_3\}$
and $\{v_1,v_2,v_3\}$ for $\mathfrak{g}_1$ and $\mathfrak{g}_2$, respectively,  as follows.
First, we choose vectors $u_1$ and $v_1$ from $\mathfrak{t}'\cap\mathfrak{g}_1$ and $\mathfrak{t}'\cap\mathfrak{g}_2$
with length $c_1$ and $c_2$, respectively. Then we freely choose any vectors
$u_2$ of length $c_1$ from $\mathfrak{t}'^\perp\cap\mathfrak{g}_1$ and set
$u_3=[u_1,u_2]$. By (2) and (3) in the lemma, we can find a vector of
$\mathfrak{h}'$ from $u_2+\mathfrak{g}_2\cap\mathfrak{t}'^\perp$. Then its
$\mathfrak{g}_2$-factor is not $0$, which can be positively scaled to a vector $v_2$ with
the length $c_2$. Then $v_1, v_2, v_3=[v_1,v_2]$ form a standard basis for $\mathfrak{g}_2$.

Now suppose $\mathfrak{h}'$ is linearly spanned by $u_1+av_1$, $u_2+bv_2$
and their bracket can be expressed as $$[u_1+av_1,u_2+bv_2]=u_3+ab v_3,$$ where $a$ is a fixed nonzero constant and
$b>0$. As a Lie algebra, $\mathfrak{h}'$ contains $[u_2+bv_2,u_3+abv_3]=u_1+ab^2 v_1$, hence
$b=1$.

With $\mathfrak{h}'$ changed to $\mathfrak{h}''$,
the same argument above can also give standard bases $\{u'_1,u'_2,u'_3\}$
and $\{v'_1,v'_2,v'_3\}$ for $\mathfrak{g}_1$ and $\mathfrak{g}_2$, respectively, such that $u'_i=u_i$ for each $i$, and $v'_1=v_1$. Then it is
easy to see that there exists a real number $t$ such that
$\mathrm{Ad}(\exp(tv_1))$ maps $v_2$ to $v'_2$, $v_3$ to $v'_3$, and keep $v_1$ and all the vectors $u_i$ unchanged. So it maps $\mathfrak{h}'$ isomorphically
to $\mathfrak{h}''$.
\end{proof}

By Lemma \ref{diagonal-A-1-conjugation-lemma}, it is easy to see that,
up to the
$\mathrm{Ad}(\exp\mathfrak{t})$-actions which preserve all the roots and
root planes of $\mathfrak{g}$,
$\mathfrak{h}_{\pm\alpha'}$ is uniquely determined. So $\mathfrak{h}$ is conjugate to
the  standard subalgebra $\mathfrak{sp}(2)$ in $\mathfrak{su}(4)$ which makes
$G/H$ a symmetric space. Since $A_3=D_3$, $G/H$ is equivalent
to the standard Riemannian sphere $S^5=\mathrm{SO}(6)/\mathrm{SO}(5)$ with constant positive curvature.

In this subcase, we can also directly prove that
$G/H$ is a symmetric homogeneous space, that is,  $[\mathfrak{m},\mathfrak{m}]\subset\mathfrak{h}$, and then apply the classification of symmetric homogeneous spaces to get the classification. However, this
argument is not valid for some other subcases below.

{\bf Subcase 2.}\quad $n=4$, and $\alpha=e_1-e_4$, $\beta=e_3-e_2$.

In this case, we have $\mathfrak{t}\cap\mathfrak{m}=\mathbb{R}(e_1+e_2-e_3-e_4)$, and it is easily seen that
$\alpha'=\frac12(e_1-e_2+e_3-e_4)$ is a unit root of $\mathfrak{h}$.
Notice that
$\mathrm{pr}_\mathfrak{h}(e_i-e_5)$, $1\leq i\leq 4$,  can not be a root of $\mathfrak{h}$
 since it is not orthogonal to $\alpha'$ and its length is $\frac{\sqrt{7}}2$. Thus any root of $\mathfrak{h}$ must be of the form
 $\mathrm{pr}_\mathfrak{h}(e_i-e_j)$ with $1\leq i<j\leq 4$.
 A similar argument as in Subcase 1 then shows that the root system of $\mathfrak{h}$ is of type $B_2=C_2$, i.e., up to the $\mathrm{Ad}(\mathrm{SU}(4))$-actions,
$\mathfrak{h}=\mathbb{R}(e_1+e_2+e_3+e_4-4e_5)\oplus\mathfrak{h}'$, where $\mathfrak{h}'$ is
the standard subalgebra $\mathfrak{sp}(2)$ in $\mathfrak{su}(4)$ corresponding to $e_i$
with $1\leq i\leq 4$. So $G/H$ is equivalent to the Berger's space
$\mathrm{SU}(5)/\mathrm{Sp}(2)\mathrm{U}(1)$, which admits positive curved normal homogeneous (Riemannian) metrics.

{\bf Subcase 3.}\quad $n>4$, and $\alpha=e_1-e_4$, $\beta=e_3-e_2$.

We have  $\mathfrak{t}\cap\mathfrak{m}=\mathbb{R}(e_1+e_2-e_3-e_4)$, and it is easily seen that
$\alpha'=\frac12(e_1-e_2+e_3-e_4)$ is a unit root of $\mathfrak{h}$.
Then it is easy to check that the roots $\gamma_1=e_1-e_5$ and $\gamma_2=e_2-e_6$
satisfy the conditions (1)-(4) of Lemma \ref{key-lemma-2}, hence the corresponding coset space does not admit any invariant reversible Finsler metric with positive flag curvature.

\subsection{The case $\mathfrak{g}=B_n$ with $n>1$}

We only need to consider the following subcases.

{\bf Subcase 1.}\quad $\alpha=e_1+e_2$, $\beta=e_2$.

In this case,  $\mathfrak{t}\cap\mathfrak{m}=\mathbb{R}e_1$ and $\alpha'=e_2$ is a root
of $\mathfrak{h}$,  with
$$\mathfrak{h}_{\pm e_2}\subset\hat{\mathfrak{g}}_{\pm e_2}=\mathfrak{g}_{\pm(e_2-e_1)}+\mathfrak{g}_{\pm e_2}+\mathfrak{g}_{\pm (e_2+e_1)}.$$
Denote $\mathfrak{g}'=\mathbb{R}e_1+\mathbb{R}e_2+
\sum_{a,b}\mathfrak{g}_{\pm(a e_1+b e_2)}$ and
$\mathfrak{g}''=\mathbb{R}e_1+\mathfrak{g}_{\pm e_1}$.
Then $\mathfrak{g}'$, $\mathfrak{g}''$ are Lie algebras of types
$B_2=\mathfrak{so}(5)$ and $A_1$, respectively.
The subalgebra $\mathfrak{h}\cap\mathfrak{g}'$ of type $A_1$ is linearly spanned by
$$
u=\left(
    \begin{array}{ccccc}
      0 & 0 & 0 & 0 & 0 \\
      0 & 0 & 0 & 0 & 0 \\
      0 & 0 & 0 & 0 & 0 \\
      0 & 0 & 0 & 0 & -1 \\
      0 & 0 & 0 & 1 & 0 \\
    \end{array}
  \right)\in\mathbb{R}e_2,\quad
v=\left(
    \begin{array}{ccccc}
      0 & 0 & 0 & -a & -a' \\
      0 & 0 & 0 & -b & -b' \\
      0 & 0 & 0 & -c & -c' \\
      a & b & c & 0 & 0 \\
      a' & b' & c' & 0 & 0 \\
    \end{array}
  \right)\in\mathfrak{h}_{\pm e_2},
$$
and
$$
w=[u,v]=\left(
          \begin{array}{ccccc}
            0 & 0 & 0 & a' & -a \\
            0 & 0 & 0 & b' & -b \\
            0 & 0 & 0 & c' & -c \\
            -a' & -b' & -c' & 0 & 0 \\
            a & b & c & 0 & 0 \\
          \end{array}
        \right),
$$
in which $(a,b,c,a',b',c')$ is a nonzero vector in $\mathbb{R}^6$.
Since  $[v,w]\in\mathfrak{h}\cap\mathfrak{g}'$,
$(a,b,c)$ and $(a',b',c')$ are linearly dependent vectors. Using a suitable
isomorphism $l\in\mathrm{Ad}(\exp\mathfrak{g}'')$ of $\mathfrak{g}$, we can make $b=c=b'=c'=0$, i.e.,  up to equivalence, we can assume that
$\mathfrak{h}_{\pm e_2}=\mathfrak{g}_{\pm e_2}$. Thus
$\mathfrak{g}_{\pm (e_2\pm e_1)}\in\mathfrak{m}$.

By Lemma \ref{key-lemma-1},
any root $\pm e_i\pm e_j$ of $\mathfrak{g}$ with $1<i<j$ must be a root of $\mathfrak{h}$ and we have $\mathfrak{h}_{\pm(e_i\pm e_j)}=\mathfrak{g}_{\pm(e_i\pm e_j)}=\hat{\mathfrak{g}}_{\pm(e_i\pm e_j)}$. By the linear isomorphism  $\mathrm{ad}(w)$ between $\hat{\mathfrak{g}}_{\pm e_2}$ and $\hat{\mathfrak{g}}_{\pm e_i}$,
for any nonzero vector $w\in\mathfrak{g}_{\pm(e_2-e_i)}$ with $i>2$, we have $\mathfrak{g}_{\pm e_i}\subset\mathfrak{h}$. Moreover, for any $i\geq 2$, we have
$\mathfrak{g}_{\pm(e_i\pm e_1)}\subset\mathfrak{m}$.
To summarize, we have
\begin{equation}\label{0009}
\mathfrak{m}=\mathbb{R}e_1+\mathfrak{g}_{\pm e_1}+\sum_{i=2}^n(\mathfrak{g}_{\pm(e_i+e_1)}+
\mathfrak{g}_{\pm(e_i-e_1)}).
\end{equation}

Let $\{u,u'\}$ be a bi-invariant orthonormal basis of $\mathfrak{g}_{\pm (e_1+e_2)}$ and choose a nonzero vector
$v$ from $\mathfrak{g}_{\pm (e_1-e_2)}$ such that
$\langle u',v\rangle_u^F=0$.
 Since the Minkowski norm $F|_{\mathfrak{g}_{\pm(e_1+e_2)}}$ is $\mathrm{Ad}(\exp(\mathbb{R}e_2))$-invariant, it coincides with the restriction of the bi-invariant inner product up to scalar changes.
So we have
\begin{equation}\label{3599}
\langle u',u\rangle_u^F=\langle[u,e_1],u\rangle_u^F=\langle[u,e_2],u\rangle_u^F=0.
\end{equation}

Now a direct calculation shows that
\begin{equation*}
[u,\mathfrak{m}]_\mathfrak{m}\subset
\mathbb{R}[e_1,u]+\mathbb{R}e_1\subset\mathbb{R}u'
+\hat{\mathfrak{g}}_{0}.
\end{equation*}
So by Lemma \ref{lemma-3-6},
$$\langle v,\hat{\mathfrak{g}_0}\rangle_u^F=\langle u,\hat{\mathfrak{g}}_0\rangle_u^F=0.$$
By our assumptions on $u$ and $v$, we have
\begin{equation}\label{3590}
\langle[u,\mathfrak{m}]_\mathfrak{m},u\rangle_u^F
=\langle\mathbb{R}u',u\rangle_u^F = 0,
\end{equation}
and
\begin{equation}\label{3601}
\langle[u,\mathfrak{m}]_\mathfrak{m},v\rangle_u^F
=\langle\mathbb{R}u',u\rangle_u^F+
\langle\hat{\mathfrak{g}}_0,v\rangle_u^F=0.
\end{equation}

Since $e_2\in\mathfrak{h}$,
by Theorem 1.3 of \cite{DH2004}, we have
\begin{equation}\label{3602}
\langle[e_2,v],u\rangle^{F}_u
	=-\langle[e_2,u],v\rangle^{F}_u-2C_u^{F}(u,v,[e_2,u]).
\end{equation}
By (\ref{3599}), the first term of the right side of above equation vanishes. By the property of Cartan tensor, $C^F_u(u,\cdot,\cdot)\equiv 0$, so the second term also vanishes.
Thus we have
\begin{equation}\label{3603}
\langle [e_1,v],u\rangle_u^F=\langle[e_2,v],u\rangle_u^F=0.
\end{equation}
A direct calculation then shows that
$$[v,\mathfrak{m}]_\mathfrak{m}\subset \mathbb{R}[e_1,v]+\hat{\mathfrak{g}}_0.$$
So by Lemma \ref{lemma-3-6} and (\ref{3603}), we have
\begin{equation}\label{3604}
\langle[v,\mathfrak{m}]_\mathfrak{m},u\rangle_u^F=
\langle\mathbb{R}[e_1,v], u\rangle_u^F+\langle\hat{\mathfrak{g}}_0,u\rangle_u^F=0.
\end{equation}

 Taking the summation of (\ref{3590}), (\ref{3601}) and (\ref{3604}), we get $U(u,v)=0$. Hence by Theorem \ref{flag-curvature-formula-thm}, we have $K^F(o,u,u\wedge v)=0$. Therefore the corresponding coset space does not admit any invariant reversible Finsler metric with positive flag curvature.

{\bf Subcase 2.}\quad $\alpha=e_1+e_2$, $\beta=e_2-e_1$.

This subcase has been covered by Subcase 1.

{\bf Subcase 3.}\quad $n=4$, and $\alpha=e_1+e_2$,  $\beta=-e_3-e_4$.

In this case, we have $\mathfrak{t}\cap\mathfrak{m}=\mathbb{R}(e_1+e_2+e_3+e_4)$, and it is easily seen that
$\alpha'=\frac12(e_1+e_2-e_3-e_4)$ is a root of $\mathfrak{h}$ with
$\mathfrak{h}_{\pm\alpha'}\subset\hat{\mathfrak{g}}_{\pm\alpha'}
=\mathfrak{g}_{\pm(e_1+e_2)}+\mathfrak{g}_{\pm(e_3+e_4)}$. The argument here is very similar to Subcase 1 for $A_n$. Obviously
$\mathfrak{h}_{\pm\alpha'}$ is not a root plane of $\mathfrak{g}$. By Lemma \ref{key-lemma-1}, if $1\leq i<j\leq 4$,
then the root $e_i-e_j$ of $\mathfrak{g}$ is also a root of $\mathfrak{h}$ with
$\mathfrak{h}_{\pm(e_i-e_j)}=\mathfrak{g}_{\pm(e_i-e_j)}=
\hat{\mathfrak{g}}_{\pm(e_i-e_j)}$. Using the action $\mathrm{ad}(u)$,
one easily shows that for any non zero vector $u\in\mathfrak{g}_{\pm(e_i-e_j)}\subset\mathfrak{h}$, with $(i,j)=(2,3)$ or $(2,4)$, both $\beta'=\frac12(e_1+e_3-e_2-e_4)$ and
$\gamma'=\frac12(e_1+e_4-e_2-e_3)$ are also roots of $\mathfrak{h}$.
Hence  $\mathfrak{h}$ is of type $B_3$, and it is uniquely  determined by
the choice of $\mathfrak{h}_{\pm\alpha'}$. By Lemma \ref{diagonal-A-1-conjugation-lemma}, up to the $\mathrm{Ad}(G)$-action,
we can assume $\mathfrak{h}$  to be the standard subalgebra such that the pair $(\mathfrak{g},\mathfrak{h})$ defines the
homogeneous Finsler sphere
$S^{15}=\mathrm{Spin}(9)/\mathrm{Spin}(7)$. So in this subcase $(G/H,F)$ must be equivalent to the homogeneous sphere $S^{15}=\mathrm{Spin}(9)/\mathrm{Spin}(7)$ on which there exist positively curved homogeneous Riemannian metrics.

{\bf Subcase 4.} $n>4$, and $\alpha=e_1+e_2$, $\beta=-e_3-e_4$.

 We have $\mathfrak{t}\cap\mathfrak{m}=\mathbb{R}(e_1+e_2+e_3+e_4)$, and it is easily seen that $\alpha'=\frac12(e_1+e_2-e_3-e_4)$ is a unit root of $\mathfrak{h}$.
Then the roots $\gamma_1=e_1+e_5$
and $\gamma_2=e_1-e_5$  satisfy (1)-(4) of Lemma \ref{key-lemma-2}. Hence the corresponding coset space does not admit any invariant reversible Finsler metric with positive flag curvature.

{\bf Subcase 5.} $n=3$, and $\alpha=e_1+e_2$, $\beta=-e_3$.

It is easily seen that $\mathfrak{t}\cap\mathfrak{m}=\mathbb{R}(e_1+e_2+e_3)$ and $\alpha'=\frac13(e_1+e_2-2e_3)$ is a root of $\mathfrak{h}$.
The argument here
is similar to that of Subcase 3. By Lemma \ref{key-lemma-1} and Lemma \ref{trick-lemma-0}, the root system of $\mathfrak{h}$ contains the
roots
$$\pm(e_i-e_j), \,\,1\leq i<j\leq 3, $$
 and$$\frac13(e_1+e_2+e_3)-e_i,\,\,1\leq i\leq 3.$$
The subalgebra $\mathfrak{h}$ is of type $\mathfrak{g}_2$,  and  is uniquely
determined by the choice of
$$\mathfrak{h}_{\pm\alpha'}\subset\hat{\mathfrak{g}}_{\pm\alpha'}=
\mathfrak{g}_{\pm(e_1+e_2)}+\mathfrak{g}_{\pm e_3}.$$
 By Lemma \ref{diagonal-A-1-conjugation-lemma},  up to the $\mathrm{Ad}(G)$-action, there exists a unique $\mathfrak{h}$, and the corresponding coset space is  the homogeneous sphere
$S^7=\mathrm{Spin}(7)/\mathrm{G}_2$. Notice that in this case
the isotropy action
is transitive, so any homogeneous Finsler metric on it must be Riemannian with positive constant curvature. Consequently in this subcase $(G/H,F)$ is equivalent to
the Riemannian homogeneous sphere $S^7=\mathrm{Spin}(7)/\mathrm{G}_2$ of positive constant curvature.

{\bf Subcase 6.}\quad  $n>3$, and $\alpha=e_1+e_2$,  $\beta=-e_3$.

In this case,  $\mathfrak{t}\cap\mathfrak{m}=\mathbb{R}(e_1+e_2+e_3)$ and
$\alpha'=\frac13(e_1+e_2-2e_3)$ is a root of $\mathfrak{h}$.
The roots  $\gamma_1=e_1+e_4$ and $\gamma_2=e_1-e_4$
satisfy the conditions (1)-(4) of Lemma \ref{key-lemma-2}. Hence the corresponding coset space does not admit any invariant reversible
Finsler metric with positive flag curvature.

{\bf Subcase 7.}\quad  $\alpha=e_1$, $\beta=e_2$.

In this case, $\mathfrak{t}\cap\mathfrak{m}=\mathbb{R}(e_1-e_2)$ and $\alpha'=\frac12(e_1+e_2)$ is a root of $\mathfrak{h}$.
By Lemma \ref{key-lemma-1}, $2\alpha'=e_1+e_2$
is also a root of $\mathfrak{h}$. Hence the corresponding coset space does not admit any invariant reversible Finsler metric with positive flag curvature.

{\bf Subcase 8.}\quad  $n=2$, and $\alpha=e_1+e_2$,  $\beta=-e_1$.

 In this case,  $\mathfrak{t}\cap\mathfrak{m}=\mathbb{R}(2e_1+e_2)$ and $\alpha'=-\frac15e_1+\frac25e_2$ is a root of $\mathfrak{h}$.
The subalgebra
$\mathfrak{h}$ is of type $A_1$, and  is uniquely  determined by the choice
of $\mathfrak{h}_{\pm\alpha'}$ in $\hat{\mathfrak{g}}_{\pm\alpha'}=\mathfrak{g}_{\pm(e_1+e_2)}+
\mathfrak{g}_{\pm e_1}$. By Lemma \ref{diagonal-A-1-conjugation-lemma},
 $G/H$ is uniquely determined up to equivalence, i.e., it is equivalent to  the Berger's space
$\mathrm{Sp}(2)/\mathrm{SU}(2)$. Hence there exists positively curved normal homogeneous Riemannian metrics on it.

{\bf Subcase 9.}\quad  $n>2$, and $\alpha=e_1+e_2$,  $\beta=-e_1$.

In this case,  $\mathfrak{t}\cap\mathfrak{m}=\mathbb{R}(2e_1+e_2)$ and $\alpha'=-\frac15 e_1+\frac25 e_2$ is
a root of $\mathfrak{h}$.
The roots  $\gamma_1=e_1+e_3$ and $\gamma_2=e_1-e_3$
satisfy (1)-(2) but does not satisfy (3) of Lemma \ref{key-lemma-2}, i.e.,
$\pm\gamma_1$ are the only roots of $\mathfrak{g}$ in $\mathbb{R}\gamma_1+\mathfrak{t}\cap\mathfrak{m}$,
and all the roots of $\mathfrak{g}$ in $\pm\gamma_2+\mathbb{R}\gamma_1+\mathfrak{t}\cap\mathfrak{m}$ are
$\pm\gamma_2=\pm(e_1-e_3)$ and $\pm\gamma_3=\pm e_2$.
Choosing $u$ and $v$ from $\mathfrak{g}_{\pm\gamma_1}$ and $\mathfrak{g}_{\pm\gamma_2}$ as in the proof for Lemma \ref{key-lemma-2},
 we can similarly get
\begin{equation}\label{3700}
\langle[u,\mathfrak{m}]_{\mathfrak{m}},u\rangle_u^F=\langle[v,\mathfrak{m}]_\mathfrak{m},u\rangle_u^F=0.
\end{equation}
Notice that $\gamma_1=e_1+e_3$ and $\gamma_3=e_2$ also satisfy
(1) of Lemma \ref{key-lemma-2}, i.e., $\gamma_1\pm\gamma_3$ are not
roots of $\mathfrak{g}$. This fact,   together with Lemma \ref{lemma-3-8}, implies that
\begin{equation}\label{3701}
\langle[u,\mathfrak{m}]_\mathfrak{m},v\rangle_u^F=0.
\end{equation}
Taking the summation of (\ref{3700}) and (\ref{3701}), we get $U(u,v)=0$. Thus by Theorem \ref{flag-curvature-formula-thm}, we have
$K^F(o,u,u\wedge v)=0$.  Hence there does not exists any invariant reversible Finsler metric on the corresponding coset space with positive flag curvature.

\subsection{The case $\mathfrak{g}=C_n$ with $n>2$}

We only need to consider the following subcases.

{\bf Subcase 1.}\quad  $\alpha=2e_1$, $\beta=e_1+e_2$.

 In this case, $\mathfrak{t}\cap\mathfrak{m}=\mathbb{R}(e_1-e_2)$ and $\alpha'=\beta=e_1+e_2$ is a root of $\mathfrak{h}$. Let $\mathfrak{t}'$ be the subalgebra of
$\mathfrak{t}\cap\mathfrak{h}$ spanned by $\{e_3,\ldots,e_n\}$, and $T'$ the
corresponding sub-torus in $T\cap H$. Then the  Lie algebra of $C_G(T')$ is
 $\mathfrak{t}'\oplus\mathfrak{g}''$, in which
 $\mathfrak{g}''$ is of type $B_2$. If the corresponding coset space can be positively curved,
 then Lemma \ref{totally-geodesic-lemma} implies that the   positively curved
 reversible homogeneous Finsler space $\mathrm{SO}(5)/\mathrm{SO}(3)$ should appear in Subcase 1 for $B_n$, which is a contradiction. Hence there does not exists any invariant reversible Finsler metric on the corresponding coset space with positive flag curvature.

{\bf Subcase 2.}\quad $\alpha=2e_1$, $\beta=2e_2$.

This subcase has been covered by the previous one.

{\bf Subcase 3.}\quad  $\alpha=2e_1$, $\beta=-e_2-e_3$.

In this case,
$\mathfrak{t}\cap\mathfrak{m}=\mathbb{R}(2e_1+e_2+e_3)$ and
$\alpha'=\frac23 e_1-\frac23 e_2-\frac23 e_3$ is a root of $\mathfrak{h}$.
Then the roots $\gamma_1=2e_2$ and $\gamma_2=2e_3$ satisfy the conditions  (1)-(4) of Lemma \ref{key-lemma-2}.  Hence there does not exists any invariant reversible Finsler metric on the corresponding coset space with positive flag curvature.

{\bf Subcase 4.}\quad  $\alpha=e_1+e_2$, $\beta=e_1-e_2$.

In this case,  $\mathfrak{t}\cap\mathfrak{m}=\mathbb{R}e_2$ and $\alpha'=e_1$ is
a root of $\mathfrak{h}$. By Lemma \ref{key-lemma-1}, $2\alpha'=2e_1$
is also a root of $\mathfrak{h}$. Hence there does not exists any invariant reversible Finsler metric on the corresponding coset space with positive flag curvature.

{\bf Subcase 5.}\quad  $\alpha=e_1+e_2$, $\beta=-e_3-e_4$.

In this case,
$\mathfrak{t}\cap\mathfrak{m}=\mathbb{R}(e_1+e_2+e_3+e_4)$
and $\alpha'=\frac12(e_1+e_2-e_3-e_4)$ is a root of $\mathfrak{h}$.
Then the  roots  $\gamma_1=2e_1$ and $\gamma_2=2e_2$ satisfy (1)-(4) of Lemma \ref{key-lemma-2}.
Hence there does not exists any invariant reversible Finsler metric on the corresponding coset space with positive flag curvature.

{\bf Subcase 6}\quad  $n>2$, and $\alpha=2e_1$, $\beta=-e_1-e_2$.

In this case,  $\mathfrak{t}
\cap\mathfrak{m}=\mathbb{R}(3e_1+e_2)$ and $\alpha'=\frac15 e_1-\frac35 e_2$
is a root of $\mathfrak{h}$. Then the roots
$\gamma_1=e_1+e_3$ and $\gamma_2=2e_2$ satisfy the conditions (1)-(4) of Lemma \ref{key-lemma-2}.
Hence there does not exists any invariant reversible Finsler metric on the corresponding coset space with positive flag curvature.

\subsection{The case $\mathfrak{g}=D_n$ with $n>3$}

We only need to consider the following subcases.

{\bf Subcase 1.} $\alpha=e_1+e_2$, $\beta=e_2-e_1$.

In this case, $\mathfrak{t}\cap\mathfrak{m}=\mathbb{R}e_1$ and $\alpha'=e_2$ is a root of $\mathfrak{h}$.
Then  we can apply Lemma \ref{key-lemma-1},
Lemma \ref{trick-lemma-0} and a similar argument as in  Subcase 1 for $A_n$ (which in fact is a special situation of this subcase), to show that $\mathfrak{h}$ is of type $B_{n-1}$ with all the roots given
by $$\pm e_i\pm e_j\mbox{ for }1<i<j\leq n\mbox{ and }\pm e_i\mbox{ for } 1<i\leq n.$$
Using Lemma \ref{diagonal-A-1-conjugation-lemma}, we can show that, up to $\mathrm{Ad}(G)$-actions, $\mathfrak{h}$ is
the  standard subalgebra such that the homogeneous Finsler space $(G/H,F)$ is equivalent to the Riemannian symmetric
sphere $\mathrm{SO}(2n)/\mathrm{SO}(2n-1)$ of positive constant curvature.

{\bf Subcase 2.} $n\ge 4$, and $\alpha=e_1+e_2$, $\beta=-e_3-e_4$.

First notice that  $D_4$ has outer automorphisms. So the argument
in the above subcase can be applied to this case. If $n>4$,
then
$\mathfrak{t}\cap\mathfrak{m}=\mathbb{R}(e_1+e_2+e_3+e_4)$ and
$\alpha'=\frac12(e_1+e_2-e_3-e_4)$ is a root of $\mathfrak{h}$ with  length $1$. Then the roots $\gamma_1=e_1+e_5$ and $\gamma_2=e_1-e_5$ satisfy the conditions
(1)-(4) of Lemma \ref{key-lemma-2}. Hence there does not exists any invariant reversible Finsler metric on the corresponding coset space with positive flag curvature.

\section{Case III: the exceptional groups and summary}

We continue  the case by case discussion of the last section, and summarize all the results of
these two sections as a theorem at the end, which is one of the main results of this paper.

\subsection{The case $\mathfrak{g}=E_6$}

Without losing generality, we can  assume that the orthogonal pair of the roots $\alpha$ and $\beta$
  are of the form $\pm e_i\pm e_j$. Up to the Weyl group action induced by $D_5$, there are two subcases: (1) $\alpha=e_1+e_2$ and $\beta=e_2-e_1$; (2)
$\alpha=e_1+e_2$ and $\beta=-e_3-e_4$.
Using the outer automorphisms of $E_6$ as well as the Weyl group action, the second subcase can be reduced to the first one. So we can assume that $\alpha=e_1+e_2$ and $\beta=e_2-e_1$. Then $\mathfrak{t}\cap\mathfrak{m}=\mathbb{R}e_1$, and $\alpha'=e_2$ is a root
of $\mathfrak{h}$. Then the roots
$$
\gamma_1=
-\frac12 e_1+\frac12 e_2+\frac 12 e_3+\frac 12 e_4+\frac12 e_5+\frac{\sqrt{3}}{2}e_6,
 $$
  and
$$\gamma_2= -\frac12 e_1-\frac12 e_2-\frac 12 e_3-\frac 12 e_4-\frac12 e_5+\frac{\sqrt{3}}{2}e_6$$
 satisfy the conditions (1)-(4) of Lemma \ref{key-lemma-2}. Hence there does not exists any invariant reversible Finsler metric on the corresponding coset space with positive flag curvature.

\subsection{The case $\mathfrak{g}=E_7$}

Given  an orthogonal pair of roots $\alpha$ and $\beta$ of $\mathfrak{g}$, we can  use certain
Weyl group action to change $\beta$ to $\sqrt{2}e_7$. Since $\beta$ is orthogonal to $\alpha$, $\alpha$ must be then of  the form $\pm e_i\pm e_j$ with $1\leq i<j\leq 6$. Using Weyl group actions induced by $D_6$, we can change $\alpha$ to  $e_1+e_2$ while keeping $\beta=\sqrt{2}e_7$ fixed.
So essentially there is
only one subcase, namely,  $\alpha=e_1+e_2$ and $\beta=e_2-e_1$. Then $\mathfrak{t}\cap
\mathfrak{m}=\mathbb{R}e_1$ and $\alpha'=e_2$ is a root of $\mathfrak{h}$.
Now the pair of roots
$$
\gamma_1= -\frac12 e_1+\frac12 e_2+\frac12 e_3+\frac12 e_4
+\frac12 e_5+\frac12 e_6 +\frac{\sqrt{2}}{2}e_7,$$
 and
$$\gamma_2= \frac12 e_1-\frac12 e_2-\frac12 e_3-\frac12 e_4
+\frac12 e_5+\frac12 e_6 +\frac{\sqrt{2}}{2}e_7$$
 satisfy the conditions (1)-(4) of Lemma \ref{key-lemma-2}. Hence there does not exists any invariant reversible Finsler metric on the corresponding coset space with positive flag curvature.

\subsection{The case $\mathfrak{g}=E_8$}

 Up to the  Weyl group action, we can assume that $\alpha$ and $\beta$ are of the form $\pm e_i\pm e_j$. We only need to consider the following two
subcases.

{\bf Subcase 1.}\quad  $\alpha=e_1+e_2$,  $\beta=e_2-e_1$.

In this case,  $\mathfrak{t}\cap\mathfrak{m}=\mathbb{R}e_1$ and $\alpha'=e_2$ is
a root of $\mathfrak{h}$. The the pair of roots
$$
\gamma_1 = \frac12 e_1+\frac12 e_2+\frac12 e_3+\frac12 e_4+
\frac12 e_5+\frac12 e_6+\frac12 e_7+\frac12 e_8,$$
 and
$$\gamma_2=-\frac12 e_1-\frac12 e_2-\frac12 e_3-\frac12 e_4+
\frac12 e_5+\frac12 e_6+\frac12 e_7+\frac12 e_8$$
 satisfy the conditions (1)-(4) of Lemma \ref{key-lemma-2}. Hence there does not exists any invariant reversible Finsler metric on the corresponding coset space with positive flag curvature.

{\bf Subcase 2.} $\alpha=e_1+e_2$ and $\beta=-e_3-e_4$.

In this case,  $\mathfrak{t}\cap\mathfrak{m}=\mathbb{R}(e_1+e_2+e_3+e_4)$ and
$\alpha'=\frac12(e_1+e_2-e_3-e_4)$ is a root of $\mathfrak{h}$. Then the pair of roots
$\gamma_1=e_1+e_5$ and $\gamma_2=e_2+e_6$ satisfy the conditions (1)-(4) of
Lemma \ref{key-lemma-2}. Hence there does not exists any invariant
reversible Finsler metric on the corresponding coset space with positive flag curvature.

\subsection{The case $\mathfrak{g}=F_4$}

Notice that up to the Weyl group action, any short root of $F_4$ can be changed to $e_1$. This implies that  any orthogonal pair of short roots of $F_4$ can be changed to the pairs $e_1$ and $-e_2$.
On the other hand, using the reflections induced by the roots $\frac12(\pm e_1\pm e_2\pm e_3\pm e_4)$, any orthogonal pair of long roots can be changed to the pair
$e_1\pm e_2$. Hence we only need to consider the following subcases.

{\bf Subcase 1.}\quad  $\alpha=e_1+e_2$, $\beta=e_2$.

In this case,  $\mathfrak{t}\cap\mathfrak{m}=\mathbb{R}e_1$ and $\alpha'=e_2$ is a root of
$\mathfrak{h}$. Let $\mathfrak{t}'$ be the subalgebra of $\mathfrak{t}\cap\mathfrak{h}$ spanned by $e_3$ and $e_4$, and $T'$ be
the closed sub-torus in $T\cap H$ with $\mathrm{Lie}(T')=\mathfrak{t}'$.
Then applying Lemma \ref{totally-geodesic-lemma} to $T'$, we conclude that there should be  a positively curved reversible
homogeneous Finsler space $\mathrm{SO}(5)/\mathrm{SO}(3)$  in Subcase 1 for $B_n$,
which is a contradiction. Hence there does not exists any invariant reversible
Finsler metric on the corresponding coset space with positive flag curvature.

{\bf Subcase 2.}\quad  $\alpha=e_1+e_2$, $\beta=e_2-e_1$.

This subcase has
been covered by the previous one.

{\bf Subcase 3.}\quad $\alpha=e_1+e_2$, $\beta=-e_3$.

In this case, $\mathfrak{t}\cap
\mathfrak{m}=\mathbb{R}(e_1+e_2+e_3)$ and $\alpha'=\frac13 e_1+\frac13 e_2-
\frac23 e_3$ is a root of $\mathfrak{h}$ with length $\sqrt{\frac23}$, with $\mathfrak{h}_{\pm\alpha'}\subset
\hat{\mathfrak{g}}_{\pm\alpha'}=\mathfrak{g}_{\pm(e_1+e_2)}+\mathfrak{g}_{\pm e_3}$. By Lemma \ref{key-lemma-1}, $\pm e_4$ are roots of $\mathfrak{h}$,
and $\mathfrak{h}_{\pm e_4}=\mathfrak{g}_{\pm e_4}=\hat{\mathfrak{g}}_{\pm e_4}$. Notice that $\mathrm{pr}_\mathfrak{h}(e_4-e_3)$ is not
orthogonal to $\alpha'$ and has  length $\sqrt{\frac53}$.
So $\mathrm{pr}_\mathfrak{h}(e_4-e_3)$ is not a root of $\mathfrak{h}$. Thus
$\mathfrak{g}_{\pm(e_4-e_3)}\subset\mathfrak{m}$. Therefore we have
$$\mathfrak{g}_{\pm e_3}=[\mathfrak{g}_{\pm e_4},\mathfrak{g}_{\pm (e_4-e_3)}]\subset\mathfrak{m}.$$
Hence $\mathfrak{h}_{\pm\alpha'}=\mathfrak{g}_{\pm(e_1+e_2)}$.
Then we have
$$\alpha'=\frac13 e_1+\frac13 e_2-\frac23 e_3
\subset[\mathfrak{h}_{\pm\alpha'},\mathfrak{h}_{\pm\alpha'}]
=[\mathfrak{g}_{\pm(e_1+e_2)},\mathfrak{g}_{\pm(e_1+e_2)}]=\mathbb{R}(e_1+e_2),$$
which is  a contradiction. Hence there does not exists any invariant
reversible Finsler metric on the corresponding coset space with positive flag curvature.

{\bf Subcase 4.}\quad  $\alpha=e_1$, $\beta=-e_2$.

 In this case, $\mathfrak{t}\cap\mathfrak{m}=\mathbb{R}(e_1+e_2)$ and $\alpha'=\frac12(e_1-e_2)$ is a root of $\mathfrak{h}$. By Lemma \ref{key-lemma-2}, $e_1-e_2=2\alpha'$ is also a root of $\mathfrak{h}$,
which is a contradiction. Hence there does not exists any invariant reversible Finsler metric on the corresponding coset space with positive flag curvature.

{\bf Subcase 5.}\quad  $\alpha=e_1+e_2$, $\beta=-e_2$.

In this case,  $\mathfrak{t}\cap\mathfrak{m}=\mathbb{R}(e_1+2e_2)$,
and $\alpha'=\frac25 e_1-\frac15 e_2$ is a root of $\mathfrak{h}$
of length $\sqrt{\frac15}$, with $\mathfrak{h}_{\pm\alpha'}\subset
\hat{\mathfrak{g}}_{\pm\alpha'}=\mathfrak{g}_{\pm(e_1+e_2)}+\mathfrak{g}_{\pm e_2}$.  By Lemma \ref{key-lemma-1}, $e_3$ is a root of $\mathfrak{h}$ and
$\mathfrak{h}_{\pm e_3}=\mathfrak{g}_{\pm e_3}=\hat{\mathfrak{g}}_{\pm e_3}$. The vector $\mathrm{pr}_\mathfrak{h}(e_2+e_3)$ is
not a root of $\mathfrak{h}$, since it is not orthogonal to $\alpha'$ and its
length is $\sqrt{\frac65}$. So $\mathfrak{g}_{\pm(e_2+e_3)}\subset\mathfrak{m}$. Then we have
$$\mathfrak{g}_{\pm e_2}=[\mathfrak{g}_{\pm(e_2+e_3)},\mathfrak{g}_{\pm e_3}]\subset\mathfrak{m}.$$
This implies that $\mathfrak{h}_{\pm\alpha'}=\mathfrak{g}_{\pm(e_1+e_2)}$.
Then we can deduce a contradiction by a similar argument as in Subcase 3 of this section.

There is another way to deduce the contradiction. Let $\mathfrak{t}'=\mathbb{R}e_4$ and $T'$ be the corresponding closed one-parameter subgroup in $H$. Using Lemma \ref{totally-geodesic-lemma}, we get a positively curved reversible homogeneous Finsler space in Subcase 9 for $B_n$, which is impossible.

\subsection{The case $\mathfrak{g}=G_2$}

If the angle between $\alpha$ and $\beta$ is $\frac \pi 6$ or $\frac \pi 2$, we can
find a pair of  short roots $\alpha_1$ and $\beta_1$ of $\mathfrak{g}$,
such that the angle between $\alpha_1$ and $\beta_1$ is $\frac \pi 3$, and
$\alpha'=\mathrm{pr}_\mathfrak{h}({\alpha}_1)=
\mathrm{pr}_\mathfrak{h}({\beta}_1)$ is a root of $\mathfrak{h}$. This is  for a contradiction to Lemma \ref{lemma-999}.

Therefore we only need to consider the case that $\alpha$ is a long root, $\beta$ is a short root,
and the angle between them is $\frac{5\pi}6$,
$\alpha'=\mathrm{pr}_\mathfrak{h}(\alpha)=\mathrm{pr}_{\mathfrak{h}}
(\beta)$ is a root of $\mathfrak{h}$, and $\mathfrak{h}$ is of type $A_1$.
Let $\gamma_1=\alpha+3\beta$ and $\gamma_2=\alpha+\beta$. Select any two nonzero vectors
$u\in\mathfrak{g}_{\pm\gamma_1}$ and $v\in\mathfrak{g}_{\pm\gamma_2}$.
Then it is not hard to see that the long root $\gamma_1$ and the short root $\gamma_2$ are orthogonal to each other,
and none of $\gamma_1\pm\gamma_2$ is a root of $\mathfrak{g}$. So $u$ and $v$
are  linearly independent and commutative. Denote the anticlockwise rotation with angle $\theta$ as $R(\theta)$.
Then there exists $g\in T_H$, and suitable orthonormal bases for each of the subspaces of $\mathfrak{m}$ below, such that
\begin{eqnarray*}
\mathrm{Ad}(g)|_{\mathfrak{t}\cap\mathfrak{m}}&=&\mathrm{Id},\\
\mathrm{Ad}(g)|_{\hat{\mathfrak{g}}_{\pm\alpha'}\cap\mathfrak{m}}
&=& R(\pi/4),\\
\mathrm{Ad}(g)|_{\mathfrak{g}_{\pm(\alpha+\beta)}=
\hat{\mathfrak{g}}_{\pm 2\alpha'}}&=& R(\pi/2),\\
\mathrm{Ad}(g)|_{\mathfrak{g}_{\pm(\alpha+2\beta)}=
\hat{\mathfrak{g}}_{\pm 3\alpha'}}&=& R(3\pi/4),\\
\mathrm{Ad}(g)|_{\mathfrak{g}_{\pm(\alpha+3\beta)}=
\hat{\mathfrak{g}}_{\pm 4\alpha'}}&=& R(\pi)=-\mathrm{Id},\\
\mathrm{Ad}(g)|_{\mathfrak{g}_{\pm(2\alpha+3\beta)}=
\hat{\mathfrak{g}}_{\pm 5\alpha'}}&=& R(5\pi/4).
\end{eqnarray*}
Denote the above subspaces as $\mathfrak{m}_k$, $k=0,1,\ldots,5$, i.e.  the action of $\mathrm{Ad}(g)$ on $\mathfrak{m}_k$ is equal to $R(k\pi/4)$.
In particular,  $\mathfrak{m}_0=\mathfrak{t}\cap\mathfrak{m}$, $\mathfrak{m}_2=\mathfrak{g}_{\pm\gamma_2}$ and $\mathfrak{m}_4=\mathfrak{g}_{\pm\gamma_1}$. By Lemma \ref{lemma-3-6},
we have
\begin{equation}\label{3801}
\langle\mathfrak{m}_4,\mathfrak{m}_i\rangle_u^F=0, \quad\forall i\neq 4.
\end{equation}

For any $v'\in\mathfrak{m}_2$ and $w'\in\mathfrak{m}_i$ with $i\neq 2$, we have
\begin{eqnarray*}
\langle v',w'\rangle_u^F&=&
\langle\mathrm{Ad}(g)v',\mathrm{Ad}(g)w'\rangle_{\mathrm{Ad}(g)u}^F=
\langle R(\pi/2)v',R(i\pi/4)w'\rangle_{-u}^F\\
&=&
\langle R(\pi/2)v',R(i\pi/4)w'\rangle_{u}^F
=\langle R(\pi/2)^2 v', R(i\pi/4)^2 w'\rangle_u^F\\
&=&\langle -v',R(i\pi/2)w'\rangle_u^F =\langle v',R((i-2)\pi/2)w'\rangle_u^F.
\end{eqnarray*}
Using a  similar argument as in the proof of Lemma \ref{lemma-3-6}, we get
\begin{equation}\label{3800}
\langle \mathfrak{m}_2,\mathfrak{m}_i\rangle_u^F=0,\quad\forall i\neq 2.
\end{equation}
By the $\mathrm{Ad}(T_H)$-invariance, the Minkowski norm $F|_{\mathfrak{m}_4}$ coincides with
the restriction of the bi-invariant inner product up to a scalar change. Thus
\begin{equation}\label{3799}
\langle[u,\mathfrak{t}],u\rangle_u^F=0.
\end{equation}
Now a direct calculation shows that
\begin{equation*}
[u,\mathfrak{m}]_{\mathfrak{m}}\subset \mathfrak{m}_0+ \mathfrak{m}_1+\mathfrak{m}_3+[u,\mathfrak{t}]+\mathfrak{m}_5,
\end{equation*}
and
\begin{equation*}
[v,\mathfrak{m}]_{\mathfrak{m}}\subset \mathfrak{m}_0+\mathfrak{m}_1+\mathfrak{m}_2+
\mathfrak{m}_3+\mathfrak{m}_5.
\end{equation*}
So by (\ref{3801}), (\ref{3800}) and (\ref{3799}), we have
\begin{equation}
\langle[u,\mathfrak{m}]_\mathfrak{m},u\rangle_u^F=
\langle[u,\mathfrak{m}]_\mathfrak{m},v\rangle_u^F=
\langle[v,\mathfrak{m}]_\mathfrak{m},u\rangle_u^F=0.
\end{equation}
Hence $U(u,v)=0$. Then by Theorem \ref{flag-curvature-formula-thm}, we get
$K^F(o,u,u\wedge v)=0$, which is a contradiction. Hence there does not exists any invariant reversible Finsler metric on the corresponding coset space with positive flag curvature.

\subsection{Summary}
We now summarize all the results in Section 4 and Section 5 as the following theorem, which gives a complete classification of  odd dimensional positively curved reversible homogeneous Finsler spaces in Case III.

\begin{theorem} \label{mainthm-part-1}
Let $(G/H,F)$ be an odd dimensional positively curved reversibly homogeneous Finsler space of Case III, i.e., with respect to a bi-invariant orthogonal
decomposition $\mathfrak{g}=\mathfrak{h}+\mathfrak{m}$ for the compact Lie algebra $\mathfrak{g}$, and a fundamental Cartan subalgebra $\mathfrak{t}$,
there are roots $\alpha$ and $\beta$ of $\mathfrak{g}$ from the same simple factor, such that $\alpha\neq\pm\beta$ and
$\mathrm{pr}_\mathfrak{h}(\alpha)=\mathrm{pr}_\mathfrak{h}(\beta)=\alpha'$
is a root of $\mathfrak{h}$.
Then $(G/H,F)$ is equivalent to one of the following
the homogeneous Finsler spaces:
\begin{description}
\item{\rm (1)}\quad The odd dimensional Riemannian symmetric spheres
$S^{2n-1}=\mathrm{SO}(2n)/\mathrm{SO}(2n-1)$ with $n>2$;
\item{\rm (2)}\quad The homogeneous spheres $S^7=\mathrm{Spin}(7)/\mathrm{G}_2$ and
$S^{15}=\mathrm{Spin}(9)/\mathrm{Spin}(7)$;
\item{\rm (3)}\quad Berger's spaces $\mathrm{SU}(5)/\mathrm{Sp}(2)\mathrm{U}(1)$
and $\mathrm{Sp}(2)/\mathrm{SU}(2)$.
\end{description}
\end{theorem}

\section{The Cases II and I}

In this section we will consider odd dimensional positively curved
reversible homogeneous Finsler spaces in Cases II and  I.

\subsection{The Case II}

Let $(G/H,F)$ be an odd dimensional positively curved reversible homogeneous
Finsler space in Case II, i.e., with respect to a bi-invariant orthogonal decomposition $\mathfrak{g}=\mathfrak{h}+\mathfrak{m}$ for the compact Lie algebra $\mathfrak{g}=\mathrm{Lie}(G)$ and a fundamental Cartan subalgebra $\mathfrak{t}$, there exists two roots $\alpha$ and
$\beta$ of $\mathfrak{g}$ from different simple factors such that
$\mathrm{pr}_\mathfrak{h}(\alpha)=\mathrm{pr}_{\mathfrak{h}}(\beta)=\alpha'$
is a root of $\mathfrak{h}$. In this situation $\alpha'$ is a linear combination of $\alpha$ and $\beta$ with two nonzero coefficients. Thus
$\mathfrak{h}_{\pm\alpha'}\subset
\hat{\mathfrak{g}}_{\pm\alpha'}=
\mathfrak{g}_{\pm\alpha}+\mathfrak{g}_{\pm\beta}$ can not be a root plane
of $\mathfrak{g}$, or equivalently,  $\mathfrak{g}_{\pm\alpha}$ and $\mathfrak{g}_{\pm\beta}$ are not contained in $\mathfrak{h}$ or $\mathfrak{m}$.

First of all, we can find a direct sum decomposition
$$\mathfrak{g}=\mathfrak{g}_1\oplus\cdots\oplus\mathfrak{g}_n\oplus\mathbb{R}^m,$$
such that each $\mathfrak{g}_i$ is a simple ideal of $\mathfrak{g}$, and
$\alpha$ and $\beta$ are roots of $\mathfrak{g}_1$ and $\mathfrak{g}_2$,
respectively. Since $\mathfrak{t}\cap\mathfrak{m}=\mathbb{R}(\alpha-\beta)
\subset\mathfrak{g}_1\oplus\mathfrak{g}_2$, the abelian factor of $\mathfrak{g}$ and $\mathfrak{t}\cap\mathfrak{g}_i$ for each $i>2$ are contained in $\mathfrak{t}\cap\mathfrak{h}$. It is also obvious that for each root $\gamma$ of $\mathfrak{g}$ with $\gamma\neq\pm\alpha$, $\gamma\neq\pm\beta$ and
$\mathrm{pr}_{\mathfrak{h}}(\gamma)=\gamma'$,
$\mathfrak{g}_{\pm\gamma}=\hat{\mathfrak{g}}_{\pm\gamma'}$ is contained  either in $\mathfrak{h}$ or in $\mathfrak{m}$.

Now we prove that for any $i>2$,   $\mathfrak{g}_i$  is contained in
$\mathfrak{t}\cap\mathfrak{h}$. Since the simple factor $\mathfrak{g}_i$
can be algebraically generated by its root planes,
we only need to prove that each root plane of $\mathfrak{g}_i$ is
contained in $\mathfrak{h}$.
Let $\gamma$ be a root of $\mathfrak{g}_i$. Since $i>2$, $\gamma$ is contained in $\mathfrak{t}\cap\mathfrak{h}$ and it is the only root
of $\mathfrak{g}$ in $\gamma+(\mathfrak{t}\cap\mathfrak{m})$. By
Lemma \ref{key-lemma-1}, $\gamma$ is a root of $\mathfrak{h}$ and
$\mathfrak{g}_{\pm\gamma}=\hat{\mathfrak{g}}_{\pm\gamma}=
\mathfrak{h}_{\pm\gamma}\subset\mathfrak{h}$.

Consider the roots of $\mathfrak{g}_1$ and $\mathfrak{g}_2$.
Up to  equivalence, we can  assume that
$\mathfrak{g}=\mathfrak{g}_1\oplus\mathfrak{g}_2$.
Let
$\gamma$ be a root of $\mathfrak{g}_1$ such that $\gamma\neq\pm\alpha$.
Since it is the only root of $\mathfrak{g}_1$ contained in
$\gamma+(\mathfrak{t}\cap\mathfrak{m})$, by Lemma \ref{key-lemma-1},
if $\gamma\in\mathfrak{t}\cap\mathfrak{h}$, then $\mathfrak{g}_{\pm\gamma}
\subset\mathfrak{h}$. On the other hand,
 if $\gamma$ is not bi-invariant
orthogonal to $\alpha$, then by Lemma \ref{trick-lemma-0},
$\mathfrak{g}_{\pm\gamma}\subset\mathfrak{m}$.
The similar assertion is  valid for any root of $\mathfrak{g}_2$.

Now we claim that there does not exist two roots
$\gamma_1$ and $\gamma_2$ of $\mathfrak{g}_1$ and $\mathfrak{g}_2$, respectively, such that their root planes are contained
in $\mathfrak{m}$. In fact, otherwise we will have $\gamma_1\neq \pm\alpha$ and $\gamma_2\neq\pm\beta$. Then  $\gamma_1$ and $\gamma_2$  satisfy the conditions (1)-(4) of Lemma \ref{key-lemma-2}, which
is impossible.

Without loss of generality, we can  assume that all roots of
$\mathfrak{g}_1$ other than $\pm\alpha$ are roots of $\mathfrak{h}$. Thus
they are bi-invariant orthogonal to $\pm\alpha$. Since $\mathfrak{g}_1$ is simple,
$\mathfrak{g}_1$ is of type $A_1$ with the only roots $\pm\alpha$. Now we consider $\mathfrak{g}_2$. We first prove the following lemma
\begin{lemma} \label{lemma-6-1}
Keep the above assumptions and notation. Then there does not
exist a pair of roots $\gamma_1$ and $\gamma_2$ of $\mathfrak{g}_2$ satisfying the following conditions:
\begin{description}
\item{\rm (1)}\quad $\gamma_1\neq\pm\gamma_2$, $\gamma_1\neq\pm\beta$ and $\gamma_2\neq\pm\beta$;
\item{\rm (2)}\quad Neither $\gamma_1$ nor $\gamma_2$ is a root of $\mathfrak{h}$;
\item{\rm (3)}\quad None of $\gamma_1\pm\gamma_2$ is a root of $\mathfrak{g}$.
\end{description}
\end{lemma}
\begin{proof}
Assume conversely that there are two roots $\gamma_1$ and $\gamma_2$ of $\mathfrak{g}_2$ satisfying (1)-(3) of the lemma.
Then it is easy to see that
$\gamma_1$ is the only root in $\gamma_1+\mathbb{R}(\alpha-\beta)$. On the
other hand, if there exist some real numbers $t_1$ and $t_2$,  such that $\gamma_3=\gamma_2+t_1\gamma_1+t_2(\alpha-\beta)$ is a root of $\mathfrak{g}$ other
than $\gamma_2$, then we have $t_2\in\{-1,0,1\}$. If $t_2=0$,  then
$\gamma_3=\gamma_2+t_1\gamma_1$, with $t_1\neq 0$,  is a root of $\mathfrak{g}_2$. This is impossible,  Since $\gamma_1\pm\gamma_2$ are not roots of
$\mathfrak{g}_2$. If $t_2=\pm 1$ then $\pm\beta=t_1\gamma_1+\gamma_2$ is a
root of $\mathfrak{g}_2$ other than $\gamma_2$. Similarly
we can get a contradiction. This implies that the pair of roots $\gamma_1$ and $\gamma_2$
satisfy the conditions (1)-(4) of Lemma \ref{key-lemma-2}, which is a contradiction.
\end{proof}

Let $\mathfrak{k}$ be the subalgebra of
$\mathfrak{g}_2$ generated by $\mathfrak{g}_{\pm\beta}$ and
$\mathfrak{h}\cap\mathfrak{g}_2$. It has the same rank as $\mathfrak{g}_2$ and  can be decomposed as a direct sum
$\mathfrak{k}=A_1\oplus(\mathfrak{h}\cap\mathfrak{g}_2)$, in which the $A_1$-factor is generated by
$\mathfrak{g}_{\pm\beta}$. By Lemma \ref{lemma-6-1}, the pair $(\mathfrak{g}_2,\mathfrak{k})$ satisfies the condition (A) in \cite{Wallach1972}. Then by Proposition 6.1 of \cite{Wallach1972}, the pair $(\mathfrak{g}_2,\mathfrak{k})$ must be one of the following:
$$((A_1,A_1),
(A_2,A_1\oplus\mathbb{R})\mbox{ or }
(C_n,A_1\oplus C_{n-1}).$$
Correspondingly,  the pair $(\mathfrak{g},\mathfrak{h})$ must be one of the following:
$$(A_1\oplus A_1,A_1), (A_1\oplus A_2,A_1\oplus\mathbb{R})
\mbox{ or }(A_1\oplus C_n, A_1\oplus C_{n-1}),$$
in which the $A_1$-factor in $\mathfrak{h}$ is the diagonal subalgebra. Thus the corresponding homogeneous Finsler space is equivalent to
the symmetric homogeneous sphere $S^3=\mathrm{SO}(4)/\mathrm{SO}(3)$,
or the Wilking's space $\mathrm{SU}(3)\times\mathrm{SO}(3)/\mathrm{U}(2)$
(which coincides with the Aloff-Wallach's space $S_{1,1}$, see \cite{AW75} and \cite{Wi1999}),
or the homogeneous sphere $S^{4n-1}=\mathrm{Sp}(n)\mathrm{Sp}(1)/\mathrm{Sp}(n-1)\mathrm{Sp}(1)$.

To summarize, we have the following theorem, which gives a complete classification of odd dimensional positively curved reversible homogeneous Finsler spaces in Case II.

\begin{theorem}\label{mainthm-part-2}
Let $(G/H,F)$ be an odd dimensional positively curved reversibly homogeneous Finsler space of Case II, i.e., with respect to a bi-invariant orthogonal
decomposition $\mathfrak{g}=\mathfrak{h}+\mathfrak{m}$ for the compact Lie algebra $\mathfrak{g}=\mathrm{Lie}(G)$ and a fundamental Cartan subalgebra $\mathfrak{t}$ of $\mathfrak{g}$, there are roots $\alpha$ and $\beta$ of $\mathfrak{g}$ from different simple factors such that
$\mathrm{pr}_\mathfrak{h}(\alpha)=\mathrm{pr}_\mathfrak{h}(\beta)=\alpha'$
is a root of $\mathfrak{h}$.
Then $(G/H,F)$ is equivalent to one of the following
 homogeneous Finsler spaces:
\begin{description}
\item{\rm (1)}\quad The symmetric homogeneous sphere
$S^{3}=\mathrm{SO}(4)/\mathrm{SO}(3)$;
\item{\rm (2)}\quad  The homogeneous spheres $\mathrm{Sp}(n)\mathrm{Sp}(1)/\mathrm{Sp}(n-1)\mathrm{Sp}(1)$;
\item{\rm (3)}\quad The Wilking's space $\mathrm{SU}(3)\times\mathrm{SO}(3)/\mathrm{U}(2)$.
\end{description}
\end{theorem}

\subsection{The Case I}

Let  $(G/H,F)$ be an odd dimensional positively curved reversible
homogeneous Finsler space in Case I, i.e.,  with respect to a bi-invariant
orthogonal decomposition $\mathfrak{g}=\mathfrak{h}+\mathfrak{m}$ for the compact Lie algebra $\mathfrak{g}=\mathrm{Lie}(G)$ and a fundamental Cartan subalgebra $\mathfrak{t}$, each root plane of
$\mathfrak{h}$ is also a root plane of $\mathfrak{g}$. Keep all the relevant notation as before. The root system of $\mathfrak{h}$ is then a subset of the root system of $\mathfrak{g}$, that is,
$\Delta_\mathfrak{h}\subset\Delta_\mathfrak{g}\cap\mathfrak{h}$. For each
root $\alpha$ of $\mathfrak{g}$, we have either $\mathfrak{g}_{\pm\alpha}=\mathfrak{h}_{\pm\alpha}\subset\mathfrak{h}$ or
$\mathfrak{g}_{\pm\alpha}\subset\mathfrak{m}$.

Suppose  $\mathfrak{g}$ has the following direct sum decomposition:
\begin{equation}\label{3950}
\mathfrak{g}
=\mathfrak{g}_0\oplus\mathfrak{g}_1\oplus\cdots\oplus\mathfrak{g}_n,
\end{equation}
 where $\mathfrak{g}_0$ is  abelian and each $\mathfrak{g}_i$, $1\leq i\leq n$, is a simple ideal. Given a nonzero vector $w$ in $\mathfrak{t}\cap\mathfrak{m}$,  let $w=w_0+\cdots+w_n$ be the decomposition of $w$ with respect to (\ref{3950}).
Then it follows  from Lemma \ref{key-lemma-1} that  $\mathfrak{g}_i$ is contained in $\mathfrak{h}$ if and only if  $w_i=0$, for any $w\in \mathfrak{t}\cap\mathfrak{m}$.
Now we have the following cases:

{\bf Case 1.}\quad
There exists $w\in \mathfrak{t}\cap\mathfrak{m}$ such that $w_0\neq 0$.

We first assert that if $\alpha$ and
$\beta$ are two roots of $\mathfrak{g}$, such that none of them is a root of $\mathfrak{h}$, then
at least one of the roots  $\alpha\pm\beta$ is a   root of $\mathfrak{g}$. In fact, otherwise
the pair of roots $\alpha, \beta$ will satisfy the conditions  (1)-(4) of Lemma \ref{key-lemma-2}, which is a contradiction. Now let $\mathfrak{k}$ be the subalgebra generated by $\mathfrak{h}$ and $\mathfrak{t}$. Then we have $\mathfrak{k}=\mathfrak{h}\oplus(\mathfrak{t}\cap\mathfrak{m})$.
Let $K$ be a closed subgroup  of $G$ with $\mathrm{Lie}(K)=\mathfrak{k}$. Then
 we have $\mathrm{rk}K=\mathrm{rk}G$. This implies that the pair $(\mathfrak{g},\mathfrak{k})$ satisfies the Condition (A) in \cite{Wallach1972}.  Thus we can suppose that in the
decomposition (\ref{3950}) of $\mathfrak{g}$, the following equation holds:
$$\mathfrak{k}=\mathfrak{g}_0\oplus\mathfrak{k}_1\oplus
\mathfrak{g}_2\oplus\cdots\oplus\mathfrak{g}_n,$$
where
$$(\mathfrak{g}_1,\mathfrak{k}_1)=(A_n,A_{n-1}\oplus\mathbb{R}),
(C_n,C_{n-1}\oplus\mathbb{R}),\mbox{ or }
(A_2,\mathbb{R}\oplus\mathbb{R}).$$
Notice that in other spaces of  Wallach's list,  the abelian factor required for this situation does not appear. If $\mathfrak{g}_1=A_2$, then by Lemma \ref{key-lemma-1}, no root of $\mathfrak{g}_1$ can be contained in $\mathfrak{t}\cap\mathfrak{h}$. Thus  $(G/H,F)$ is equivalent to one of the following:
\begin{description}
\item{\rm (1)}\quad  The homogeneous sphere $$
S^{2n-1}=\mathrm{U}(n)/\mathrm{U}(n-1)
\mbox{ or }
S^{4n-1}=\mathrm{Sp}(n)\mathrm{U}(1)/\mathrm{Sp}(n-1)\mathrm{U}(1)\mbox{ for }n>1;$$
\item{\rm (2)}\quad  The $\mathrm{U}(3)$-homogeneous presentations of Aloff-Wallach's spaces $S_{k,l}=\mathrm{U}(3)/T^2$, in which $T^2$ is a two dimensional torus of diagonal matrices which does not contain the center of $\mathrm{U}(3)$ and $$T^2\cap\mathrm{SU}(3)=U_{k,l}=
\{\mathrm{diag}(z^k,z^l,z^{-k-l})|z\in\mathbb{C},|z|=1\},$$
where $k$ and $l$ are integers satisfying $kl(k+l)\neq 0$.
\end{description}

Notice that the  $\mathrm{SU}(3)$-homogeneous space $S_{k,l}$  have infinitely many different presentation as $\mathrm{U}(2)$-homogeneous spaces; See \cite{AW75}.

\medskip
{\bf Case 2.} There exists $w\in \mathfrak{t}\cap\mathfrak{m}$ with decomposition  $w=w_1+w_2$, where  both $w_1$ and $w_2$ are nonzero.

Up to equivalence, we can  assume that $\mathfrak{g}=\mathfrak{g}_1\oplus\mathfrak{g}_2$.

We first assert that  there does not exist a root $\alpha$ of $\mathfrak{g}_1$, and a root
$\beta$ of $\mathfrak{g}_2$ such that $\alpha \notin\mathbb{R}w_1$, $\beta \notin\mathbb{R}w_2$,  and none of them is a  root of $\mathfrak{h}$. In fact, otherwise  the pair of roots  $\alpha$ and $\beta$ will satisfy (1)-(4) of Lemma \ref{key-lemma-2}, which is a contradiction.
Without loss of generality, we can assume that all the roots of $\mathfrak{g}_1$
outside $\mathbb{R}w_1$ are roots of $\mathfrak{h}$, i.e., they are contained in
$\mathfrak{t}\cap\mathfrak{h}$. By the simpleness of $\mathfrak{g}_1$, we
must have $\mathfrak{g}_1=A_1$, and   the only roots  in $\mathfrak{t}\cap\mathfrak{g}_1=\mathbb{R}w_1$ are $\pm\alpha$. There are two subcases:

{\bf Subcase 1.}\quad  There exists a root $\beta$ of $\mathfrak{g}_2$ contained in $\mathbb{R}w_2$.

Obviously neither $\alpha$ nor $\beta$ is a root of $\mathfrak{h}$, i.e., their root planes are contained in $\mathfrak{m}$.
Let $\mathfrak{t}'$ be the bi-invariant orthogonal complement
of $w_2$ in $\mathfrak{g}_2$ and $T'$ be the corresponding torus in $H$.
Using Lemma \ref{totally-geodesic-lemma} for $T'$, we get a positively curved reversible homogeneous Finsler space $\mathrm{SU}(2)\times\mathrm{SU}(2)/\mathrm{U}(1)$ in Case I, in which
$\mathrm{U}(1)$ is not contained in any of the  simple factors. To prove the
reversible homogeneous space $G/H$ can not be positively curved in this subcase,
we only need to consider the  situation that
$\mathfrak{g}_2=A_1$, and the only roots are $\pm\beta$.
Fix a bi-invariant inner product on $\mathfrak{g}=\mathfrak{g}_1\oplus\mathfrak{g}_2=
A_1\oplus A_1$
such that its restriction on each factor has the same scale. By suitably re-ordering the two simple factors, we can assume that
$\alpha+c\beta\in\mathfrak{t}\cap\mathfrak{m}$ with $|c|\geq 1$.
Denote $\alpha'=\mathrm{pr}_\mathfrak{h}(\alpha)$
and $\beta'=\mathrm{pr}_\mathfrak{h}(\beta)$. Then the above  assumption implies that
$\beta'$ is not  an even multiple of $\alpha'$.

Let $\{u,u'\}$ be a bi-invariant orthonormal basis of
$\mathfrak{g}_{\pm\alpha}$, and $v$ a nonzero vector in $\mathfrak{g}_{\pm\beta}$ such that $\langle u',v\rangle_u^F=0$.
Obviously $u$ and $v$ are  linearly independent and commutative.
By the $\mathrm{Ad}(H)$-invariance, the Minkowski norm
$F|_{\mathfrak{g}_{\pm\alpha}}$ coincides with the bi-invariant inner product
up to scalar changes.  Thus
\begin{equation*}
\langle u',u\rangle_u^F=\langle[\mathfrak{t},u],u\rangle_u^F=0.
\end{equation*}
By the assumption and Lemma \ref{lemma-3-6}, we have
\begin{equation}\label{3970}
\langle \mathfrak{t}\cap\mathfrak{m},u\rangle_u^F=
\langle \mathfrak{t}\cap\mathfrak{m},v\rangle_u^F=0.
\end{equation}
Then a direct calculation shows that
\begin{equation}
[u,\mathfrak{m}]_\mathfrak{m}=\mathfrak{t}\cap\mathfrak{m}+[\mathfrak{t},u].
\end{equation}
So by (\ref{3970}), we get
\begin{equation}\label{3960}
\langle[u,\mathfrak{m}]_\mathfrak{m},u\rangle_u^F=
\langle\mathfrak{t}\cap\mathfrak{m},u\rangle_u^F+
\langle\mathbb{R}u',u\rangle_u^F=0,
\end{equation}
and
\begin{equation}\label{3961}
\langle[u,\mathfrak{m}]_\mathfrak{m},v\rangle_u^F=
\langle\mathfrak{t}\cap\mathfrak{m},v\rangle_u^F+
\langle\mathbb{R}u',v\rangle_u^F=0.
\end{equation}

Now a direct calculation  shows that
\begin{equation}\label{3850}
[v,\mathfrak{m}]_\mathfrak{m}=\mathfrak{t}\cap\mathfrak{m}+
[\mathfrak{t}\cap\mathfrak{m},v]=\mathfrak{t}\cap\mathfrak{m}+
[\mathfrak{t}\cap\mathfrak{h},v].
\end{equation}
For any $w'\in\mathfrak{t}\cap\mathfrak{h}$, we have,  by Theorem 3.1 of \cite{DH2004},
$$\langle[w',v],u\rangle_u^F=-\langle v,[w',u]\rangle_u^F-2C^F_u([w',u],v,u)=0.$$
So by Lemma \ref{lemma-3-6} and (\ref{3850}), we have
\begin{equation}\label{3962}
\langle[v,\mathfrak{m}]_\mathfrak{m},u\rangle_u^F=
\langle[v,\mathfrak{t}\cap\mathfrak{h}],u\rangle_u^F=0.
\end{equation}

 Taking the summation of (\ref{3960}), (\ref{3961}) and (\ref{3962}), we get $U(u,v)=0$. Hence  by Theorem \ref{flag-curvature-formula-thm},   $K^F(o,u,u\wedge v)=0$,
which is a contradiction. Hence the corresponding coset space does not admit any invariant Finsler metric with positive flag curvature.

\medskip
{\bf Subcase 2.}\quad
There does not exist any root of $\mathfrak{g}_2$ in $\mathbb{R}w_2$.

Then by the simpleness of $\mathfrak{g}_2$, there is a root
$\beta$ of $\mathfrak{g}_2$ which is not bi-invariant orthogonal to $w_2$.
Let $u$ and $v$ be any nonzero vectors in $\mathfrak{g}_{\pm\alpha}$ and $\mathfrak{g}_{\pm\beta}$ respectively. Then  they are  linearly independent and commutative. The  subalgebra $\mathfrak{t}'=\mathfrak{t}\cap\mathfrak{h}\cap\mathfrak{g}_2$  coincides with $w_2^{\perp}\cap
\mathfrak{t}\cap\mathfrak{g}_{2}$, the bi-invariant orthogonal complement of $w_2$ in $\mathfrak{t}\cap\mathfrak{g}_2$. Denote $T'$ the corresponding
torus in $H$.
Since the inner product $\langle\cdot,\cdot\rangle_u^F$ is
$\mathrm{Ad}(T')$-invariant, by Lemma \ref{lemma-3-8}, $\mathfrak{m}$ can be $g_u^F$-orthogonally decomposed as the sum of
$\mathfrak{m}'=\hat{\hat{\mathfrak{m}}}_0=\mathfrak{t}\cap\mathfrak{m}+\mathfrak{g}_{\pm\alpha}$ for the trivial irreducible $T'$-representation and $\mathfrak{m}''\subset\mathfrak{g}_2$ for nontrivial irreducible $T'$-representations. Notice that $\mathfrak{m}''$ is the sum of some root planes in $\mathfrak{g}_2$, and  $u$ and $v$ are contained in $\mathfrak{m}'$ and $\mathfrak{m}''$, respectively.

Now a direct calculation shows that
\begin{equation}\label{3900}
[u,\mathfrak{m}]_\mathfrak{m}=\mathfrak{t}\cap\mathfrak{m}+
[\mathfrak{t},u]\subset\mathfrak{m}'\mbox{ and }[v,\mathfrak{m}]_{\mathfrak{m}}\subset
\mathfrak{t}\cap\mathfrak{m}+\mathfrak{m}''.
\end{equation}
Moreover, the $\mathrm{Ad}(T_H)$ invariance of $F|_{\mathfrak{g}_{\pm\alpha}}$ implies that $F|_{\mathfrak{g}_{\pm\alpha}}$ coincides with the restriction of a bi-invariant inner product up to scalar changes. Thus we have
\begin{equation}\label{3901}
\langle[\mathfrak{t},u],u\rangle_u^F=
\langle[\mathfrak{t}\cap\mathfrak{h},u],u\rangle_u^F=0.
\end{equation}
By Lemma \ref{lemma-3-6},
\begin{equation}\label{3902}
\langle \mathfrak{t}\cap\mathfrak{m},u\rangle_u^F=0.
\end{equation}
Taking the summation of  (\ref{3900}), (\ref{3901}) and (\ref{3902}), we get
$$\langle[u,\mathfrak{m}]_\mathfrak{m},u\rangle_u^F=
\langle[u,\mathfrak{m}]_\mathfrak{m},v\rangle_u^F=
\langle[v,\mathfrak{m}]_\mathfrak{m},u\rangle_u^F=0.$$
Therefore $U(u,v)=0$. Now by Theorem \ref{flag-curvature-formula-thm},
$K^F(o,u,u\wedge v)=0$. This is a contradiction. Hence the corresponding coset space does not admit any invariant Finsler metric with positive flag curvature.

{\bf Case 3.}\quad There exists $w\in \mathfrak{t}\cap\mathfrak{m}$  such that $w=w_1+\cdots+w_m$,  where $m>2$ and  $w_i$, $1\leq i\leq m$ are all  nonzero.

If there is a root $\alpha\notin\mathbb{R}w_1$ of $\mathfrak{g}_1$,
and a root $\beta\notin\mathbb{R}w_2$ of $\mathfrak{g}_2$ such that they are
not roots of $\mathfrak{h}$, then they  satisfy the conditions (1)-(4) of Lemma \ref{key-lemma-2}, which is a contradiction. Similarly to the previous case, we can assume that $\mathfrak{g}_1=A_1$. Let $\pm\alpha$ be
the only roots of $\mathfrak{g}_1$, then we have
$\mathfrak{g}_{\pm\alpha}\subset\mathfrak{m}$.
We can also find a root $\beta$ of $\mathfrak{g}_2$ which is not bi-invariant
orthogonal to $w_2$, then
$\mathfrak{g}_{\pm\beta}\subset\mathfrak{m}$.
Let $u$ and $v$ be any nonzero vectors
in $\mathfrak{g}_{\pm\alpha}$ and $\mathfrak{g}_{\pm\beta}$ respectively.
Notice there does not exist any root which is contained in
$\mathbb{R}(w_2+\cdots+w_m)$, thus a similar argument as for Subcase 2 of the previous case
can be applied to prove $K^F(o,u,u\wedge v)=0$, which is a contradiction.  Hence the corresponding coset space does not admit any invariant reversible Finsler metric with positive flag curvature.

The above results can be summarized as the following theorem, which gives a
"nearly" complete classification for the coset spaces in Case I which admit invariant Finsler metrics with positive flag curvature.
\begin{theorem} \label{mainthm-part-3}
Let $(G/H,F)$ be an odd dimensional positively curved reversible homogeneous Finsler space in Case I, i.e., with respect to a bi-invariant orthogonal decomposition $\mathfrak{g}=\mathfrak{h}+\mathfrak{m}$ of the compact Lie algebra $\mathfrak{g}=\mathrm{Lie}(G)$ and a fundamental Cartan subalgebra $\mathfrak{t}$ of $\mathfrak{g}$, each root plane of $\mathfrak{h}$ is also a root plane of
$\mathfrak{g}$.
Assume that $(G/H,F)$ is not equivalent to a homogeneous Finsler space $(G'/H',F')$ with  compact simple $G'$. Then $(G/H,F)$ must be equivalent to one of the following:
\begin{description}
\item{\rm (1)}\quad The homogeneous sphere $$
S^{2n-1}=\mathrm{U}(n)/\mathrm{U}(n-1), \mbox{ or }
S^{4n-1}=\mathrm{Sp}(n)\mathrm{U}(1)/\mathrm{Sp}(n-1)\mathrm{U}(1),\quad n>1;$$
\item{\rm (2)}\quad  The $\mathrm{U}(3)$-homogeneous Aloff-Wallach's spaces.
\end{description}
\end{theorem}

\subsection{Remarks on  normal homogeneous  and Riemannian manifolds}

When we continue the classification for the odd dimensional positively curved
reversible homogeneous Finsler spaces $G/H$ in Case I with $G$ simple, we
meet both technical and substantial difficulties. However,  if we additionally assume $G/H$ to be normal homogeneous or Riemannian, then the classification can
be completed.

 If $(G/H, F)$ is normal homogeneous, then $F$ is subduced by a bi-invariant Finsler metric on $G$. Let $\mathfrak{k}$ be the subalgebra generated by
$\mathfrak{h}$ and $\mathfrak{t}$, and $K$ be the closed subgroup of $G$ with
$\mathrm{Lie}(K)=\mathfrak{k}$. Then we have $\mathfrak{k}=\mathfrak{h}\oplus(\mathfrak{t}\cap\mathfrak{m})$. The same bi-invariant Finsler metric on $G$ defines another normal homogeneous Finsler
metric $\bar{F}$ on $G/K$, such that the natural projection from $G/H$ to
$G/K$ is a Finslerian submersion. Since $\dim G/H >1$, $G/K$ is an even dimensional coset space admitting positively curved normal homogeneous Finsler metrics, which has been classified in \cite{XD2014}. From this clue, we can
easily find the missing homogeneous spheres
$$S^{2n-1}=\mathrm{SU}(n)/\mathrm{SU}(n-1) \mbox{ and }
S^{4n-1}=\mathrm{Sp}(n)/\mathrm{Sp}(n-1),\quad\forall n>1.$$

 If $(G/H, F)$ is Riemannian, then the metric is induced by an $\mathrm{Ad}(H)$-invariant inner product $\langle\cdot,\cdot\rangle$. The submersion technique described above
still works when there does not exist two different roots $\alpha$ and $\beta$
such that $\alpha-\beta\in\mathfrak{t}\cap\mathfrak{m}$. Setting $\beta=-\alpha$, the assumption also implies that there does not exist any root contained in $\mathfrak{t}\cap\mathfrak{m}$. In fact, if $\mathfrak{g}$ is simple, and there is a root contained in $\mathfrak{t}\cap\mathfrak{m}$, then we can always find roots $\alpha$ and $\beta$, such that $\alpha\neq\beta$ and
$\alpha-\beta\in\mathfrak{t}\cap\mathfrak{m}$.

Notice that the lemmas in Subsection \ref{subsection-key-lemmas} can strengthened  in Riemannian geometry. For example, Lemma \ref{key-lemma-2} can be strengthened to the following

\begin{lemma} Let $G/H$ be an odd dimensional positively curved Riemannian homogeneous space, with a bi-invariant orthogonal decomposition
$\mathfrak{g}=\mathfrak{h}+\mathfrak{m}$ and a fundamental Cartan subalgebra
$\mathfrak{t}$. Then there does not exist two roots $\alpha$ and $\beta$ of
$\mathfrak{g}$, satisfying the following conditions:
\begin{description}
\item{\rm (1)}\quad $\alpha$ and $\beta$ are not roots of $\mathfrak{h}$;
\item{\rm (2)}\quad $\alpha\pm\beta$ are not roots of $\mathfrak{g}$;
\item{\rm (3)}\quad $\alpha$ is the only roots contained in
$\alpha+\mathfrak{t}\cap\mathfrak{m}$;
\item{\rm (4)}\quad $\beta$ is the only roots contained in $\beta+\mathfrak{t}\cap\mathfrak{m}$.
\end{description}
\end{lemma}

Assuming that there exist roots $\alpha$ and $\beta$ of $\mathfrak{g}$ such that
$\alpha\neq\pm\beta$ and $\alpha-\beta\in\mathfrak{t}\cap\mathfrak{m}$, we can
use the $\mathrm{Ad}(H)$-invariance of the inner product $\langle\cdot,\cdot\rangle$ and those strengthened key lemmas, to discuss each
possible case. Notice that in the situation of $\mathrm{Sp}(2)/\mathrm{U}(1)$ in \cite{XW2015}, it is positively curved for all commutative pairs, so B. Wilking's method is essential here. The case by case discussion is not hard, but too long to be presented here. See \cite{BB76} for the original proof (with a correction in \cite{XW2015}), or
the recent paper \cite{WZ2015} for a much shorter proof.

\section{Appendix: the root systems of compact simple Lie algebras}
\label{appendix-section}

For each simple Lie algebra $\mathfrak{g}$,  we present here the Bourbaki
description of the root system $\Delta_\mathfrak{g}$, and the root planes in the classical cases.

\smallskip
(1)\quad  The case $\mathfrak{g}=A_n=\mathfrak{su}(n+1)$ for $n>0$.

Let $\{e_1,\ldots,e_{n+1}\}$ denote the standard orthonormal basis of
$\mathbb{R}^{n+1}$.  Then $\mathfrak{t}$ can be isometrically identified with
the subspace $(e_1+\cdots+c_{n+1})^\perp \subset \mathbb{R}^{n+1}$. The
root system $\Delta$ is
\begin{equation*}
\{\pm(e_i-e_j) \mid 1\leqq i < j\leqq n+1\}.
\end{equation*}
Let $E_{i,j}$ be the matrix with $1$ in the $(i,j)$-entry and all other entries zero.  Then
\begin{eqnarray*}
e_i&=&\sqrt{-1}E_{i,i}\in\mathfrak{su}(n+1), \mbox{ and} \\
\mathfrak{g}_{\pm(e_i-e_j)}&=&\mathbb{R}(E_{i,j}-E_{j,i})+\mathbb{R}\sqrt{-1}(E_{i,j}+E_{j,i}).
\end{eqnarray*}

\smallskip
(2) The case $\mathfrak{g}=B_n=\mathfrak{so}(2n+1)$ for $n>1$.

The Cartan subalgebra $\mathfrak{t}$ can be isometrically identified with
$\mathbb{R}^n$ with the standard orthonormal basis $\{e_1,\ldots,e_n\}$.
The root system $\Delta$ is
\begin{equation*}
\{\pm e_i \mid 1\leqq i\leqq n\} \cup
	\{\pm e_i\pm e_j \mid 1\leqq i<j\leqq n\}.
\end{equation*}
In terms of  matrices, we have
\begin{eqnarray*}
e_i&=&E_{2i,2i+1}-E_{2i+1,2i}, \\
\mathfrak{g}_{\pm e_i}&=&\mathbb{R}(E_{2i,1}-E_{1,2i})+\mathbb{R}(E_{2i+1,1}-E_{1,2i+1}),\\
\mathfrak{g}_{\pm(e_i-e_j)}&=&\mathbb{R}(E_{2i,2j}+E_{2i+1,2j+1}
-E_{2j,2i}-E_{2j+1,2i+1})\\
&&\phantom{X}+\mathbb{R}(E_{2i,2j+1}-E_{2i+1,2j}+E_{2j,2i+1}-E_{2j+1,2i}),
\mbox{ and}\\
\mathfrak{g}_{\pm(e_i+e_j)}&=&
\mathbb{R}(E_{2i,2j}-E_{2i+1,2j+1}-E_{2j,2i}+E_{2j+1,2i+1})\\
&&\phantom{X}+\mathbb{R}(E_{2i,2j+1}+E_{2i+1,2j}-E_{2j,2i+1}-E_{2j+1,2i}).
\end{eqnarray*}

\smallskip
(3) The case $\mathfrak{g}=C_n=\mathfrak{sp}(n)$ for $n>2$.

As before, $\mathfrak{t}$ is isometrically identified with
$\mathbb{R}^n$ with the standard orthonormal basis $\{e_1,\ldots,e_n\}$.
The root system $\Delta$ is
\begin{equation*}
\{\pm 2e_i \mid 1\leqq i\leqq n\} \cup
	\{\pm e_i\pm e_j \mid 1\leqq i<j\leqq n\}.
\end{equation*}
In terms of  matrices, we have
\begin{eqnarray*}
e_i &=&\mathbf{i}E_{i,i},\\
\mathfrak{g}_{\pm 2e_i}&=&\mathbb{R}\mathbf{j}E_{i,i}+\mathbb{R}\mathbf{k}E_{i,i},\\
\mathfrak{g}_{\pm(e_i-e_j)}&=&\mathbb{R}(E_{i,j}-E_{j,i})+\mathbb{R}\mathbf{i}(E_{i,j}+E_{j,i}),\mbox{ and}\\
\mathfrak{g}_{\pm(e_i+e_j)}&=&\mathbb{R}\mathbf{j}(E_{i,j}+E_{j,i})+
\mathbb{R}\mathbf{k}(E_{i,j}+E_{j,i}).
\end{eqnarray*}

(4) The case $\mathfrak{g}=D_n=\mathfrak{so}(2n)$ for $n>3$.

 The
Cartan subalgebra $\mathfrak{t}$ is identified with
$\mathbb{R}^n$ with the standard orthonormal basis $\{e_1,\ldots,e_n\}$.
The root system $\Delta$ is
\begin{equation*}
\{\pm e_i\pm e_j \mid 1\leqq i<j\leqq n\}.
\end{equation*}
In matrices, we have formulas for the $e_i$ and for the root planes for
$e_i\pm e_j$ similar to those in the case of $B_n$, i.e.
\begin{eqnarray*}
e_i&=&E_{2i-1,2i}-E_{2i,2i-1}, \\
\mathfrak{g}_{\pm(e_i-e_j)}&=&\mathbb{R}(E_{2i-1,2j-1}+E_{2i,2j}
-E_{2j-1,2i-1}-E_{2j,2i})\\
&&\phantom{X}+\mathbb{R}(E_{2i-1,2j}-E_{2i,2j-1}+E_{2j-1,2i}-E_{2j,2i-1}),
\mbox{ and}\\
\mathfrak{g}_{\pm(e_i+e_j)}&=&
\mathbb{R}(E_{2i-1,2j-1}-E_{2i,2j}-E_{2j-1,2i-1}+E_{2j,2i})\\
&&\phantom{X}+\mathbb{R}(E_{2i-1,2j}+E_{2i,2j-1}-E_{2j-1,2i}-E_{2j,2i-1}).
\end{eqnarray*}

\smallskip
(5) The case $\mathfrak{g}=E_6$.

The Cartan subalgebra $\mathfrak{t}$ can be isometrically identified with
$\mathbb{R}^6$ with the standard orthonormal basis $\{e_1,\ldots,e_6\}$.
The root system is
\begin{equation*}
\{\pm e_i\pm e_j \mid 1\leqq i<j\leqq 5\}
\cup \{\pm\frac12 e_1\pm\cdots\pm\frac12 e_5\pm\frac{\sqrt{3}}{2}e_6
	\mbox{ with odd number of +'s}\}.
\end{equation*}
It contains a root system of type
$D_5$.

\smallskip
(6) The case $\mathfrak{g}=E_7$.

The Cartan subalgebra
can be isometrically identified with
$\mathbb{R}^7$ with the standard orthonormal basis $\{e_1,\ldots,e_7\}$.
The root system is
\begin{eqnarray*}
& &\{\pm e_i\pm e_j \mid 1\leqq i<j<7\} \cup \{\pm\sqrt{2}e_7;
\frac12(\pm e_1\pm\cdots\pm e_6\pm\sqrt{2}e_7)\nonumber\\
& &\mbox{  with an odd number of plus signs among the first six coefficients}\}.
\end{eqnarray*}
It contains a root system of $D_6$.

\smallskip
(7) The case $\mathfrak{g}=E_8$.

The Cartan subalgebra
can be isometrically identified with
$\mathbb{R}^8$ with the standard orthonormal basis $\{e_1,\ldots,e_8\}$.
The root system $\Delta$ is
\begin{eqnarray*}
&&\{\pm e_i\pm e_j \mid 1\leqq i<j\leqq 8\} \cup\nonumber\\
&&\{\frac12(\pm e_1\pm\cdots\pm e_8) \mbox{ with an even number of +'s}\}.
\end{eqnarray*}
It contains a root system of $D_8$.

\smallskip
(8) The case $\mathfrak{g}=F_4$.

The Cartan subalgebra is
isometrically identified with $\mathbb{R}^4$ with the standard orthonormal
basis $\{e_1,\ldots,e_4\}$. The root system is
\begin{eqnarray*}
\{\pm e_i \mid 1\leqq i\leqq 4\} \cup \{\pm e_i\pm e_j \mid 1\leqq i<j\leqq 4\}
	\cup \{\frac12(\pm e_1\pm\cdots\pm e_4)\}.
\end{eqnarray*}
It contains the root system of $B_4$.

\smallskip
(9) The case $\mathfrak{g}=G_2$.

The Cartan subalgebra
is isometrically identified with $\mathbb{R}^2$ with the standard
orthonormal basis $\{e_1,e_2\}$. The root system $\Delta$ is
\begin{eqnarray*}
\{(\pm\sqrt{3},0),(\pm\frac{\sqrt{3}}{2},\pm\frac{3}{2}),
(0,\pm 1),(\pm \frac{\sqrt{3}}{2},\pm\frac{1}{2})\}.
\end{eqnarray*}

\end{document}